\documentclass[a4paper,13pt,twoside]{article}
\usepackage{amsmath,amssymb,amscd}
\usepackage{ascmac}
\usepackage{enumerate}
\usepackage{mathrsfs}
\usepackage{authblk}
\usepackage[all]{xy}
\usepackage{cancel}
\usepackage{theorem}
\newtheorem{thm}{Theorem}[section]
\newtheorem{pro}[thm]{Proposition}
\newtheorem{cor}[thm]{Corollary}
\newtheorem{lemma}[thm]{Lemma}

{\theorembodyfont{\normalfont}
\newtheorem{rmk}[thm]{Remark}
\newtheorem{rmks}[thm]{Remarks}
\newtheorem{dfn}[thm]{Definition}
\newtheorem{ex}[thm]{Example}
\newtheorem{exs}[thm]{Examples}
}
\makeatletter
\newenvironment{proof}[1][\proofname]{\par
  \normalfont
  \topsep6\p@\@plus6\p@ \trivlist
  \item[\hskip\labelsep{\textit{\mdseries #1}\@addpunct{\mdseries.}}]\ignorespaces
}{%
  \QED \endtrivlist
}
\newcommand{\proofname}{\normalfont{\textit{Proof.}}}
\makeatother
\makeatletter
\def\BOXSYMBOL{\RIfM@\bgroup\else$\bgroup\aftergroup$\fi
  \vcenter{\hrule\hbox{\vrule height.85em\kern.6em\vrule}\hrule}\egroup}
\makeatother
\newcommand{\BOX}{%
  \ifmmode\else\leavevmode\unskip\penalty9999\hbox{}\nobreak\hfill\fi
  \quad\hbox{\BOXSYMBOL}}
\newcommand\QED{\BOX}

\newcommand{\bb}[1]{\mathbb#1}
\newcommand{\fra}[1]{\mathfrak#1}
\newcommand{\ca}[1]{\mathcal#1}
\newcommand{\innpro}[3]{\langle#1,#2\rangle_{#3}}
\newcommand{\Innpro}[3]{\langle\langle#1,#2\rangle\rangle_{#3}}
\newcommand{\sine}[1]{\frac{\sin#1\theta}{\sqrt{#1\pi}}}
\newcommand{\cosi}[1]{\frac{\cos#1\theta}{\sqrt{#1\pi}}}
\newcommand{\rot}{{\rm rot}}
\newcommand{\WZW}{{\rm WZW}}
\newcommand{\alg}{{\rm alg}}
\newcommand{\ind}{{\rm ind}}
\newcommand{\id}{{\rm id}}
\newcommand{\Hom}{{\rm Hom}}
\newcommand{\op}{{\rm op}}
\newcommand{\Lie}{{\rm Lie}}
\newcommand{\fin}{{\rm fin}}
\newcommand{\ad}{{\rm ad}}

\begin{document}
\title{An Analytic $LT$-equivariant Index and Noncommutative Geometry}
\author{Doman Takata\thanks{d.takata@math.kyoto-u.ac.jp}}
\affil{Kyoto University}

\maketitle
\begin{abstract}
Let $T$ be a circle and $LT$ be its loop group. Let $\ca{M}$ be an infinite dimensional manifold equipped with a nice $LT$-action.
We construct an analytic $LT$-equivariant index for some infinite dimensional manifolds, and justify it in terms of noncommutative geometry. More precisely, we construct a Hilbert space $\ca{H}$ consisting of ``$L^2$-sections of a Clifford module bundle'' and a ``Dirac operator'' $\ca{D}$ which acts on $\ca{H}$. Then, we define an analytic index of $\ca{D}$ valued in the representation group of $LT$, so called Verlinde ring. We also define a ``twisted crossed product $LT\ltimes_\tau C_0(\ca{M})$,'' although we cannot define each concept ``function algebra for $\ca{M}$ vanishing at infinity,'' ``function from $LT$ to a $C^*$-algebra vanishing at infinity,'' and a Haar measure on $LT$. Moreover we combine all of them in terms of spectral triples and verify that the triple has an infinite spectral dimension. Lastly, we add some applications including Borel-Weil theory for $LT$.
\end{abstract}

\tableofcontents
\section{Introduction}\label{Section Introduction}
Atiyah-Singer
 index theorem is one of the greatest theorems in mathematics (\cite{Fur}). Our dream is to establish an infinite dimensional version of it. Although the project is very hard, we succeed a construction of equivariant analytic index in some special cases. We show the main result of this paper without any definitions.
\begin{thm}\label{Thm index Sect.1}
Let $T$ be a circle and $LT$ be its loop group $C^\infty(S^1,T)$,
$\ca{M}$ be an infinite dimensional manifold which $LT$ acts on ``nicely,'' $\ca{L}$ be a $\tau$-twisted $LT$-equivariant line bundle on $\ca{M}$, and $\ca{S}$ be a Spinor bundle.
Then, we can define a Hilbert space $\ca{H}:=$``$L^2(\ca{M},\ca{L}\otimes\ca{S})$'' and a ``Dirac operator'' $\ca{D}$ acting on it. Moreover, $\ca{D}$ has a well-defined analytic index valued in $R^\tau(LT)$.
\end{thm}
The readers might feel the construction is just a formal discussion. We present some circumstantial evidences that our construction may know some geometrical or analytical information of $\ca{M}$ and $\ca{L}$. The following is the biggest evidence which we know so far.
\begin{thm}\label{Thm algebra Sect.1}
Although neither a function algebra $C_0(\ca{M})$ nor a measure on $LT$ has not been defined, we can define a $C^*$-algebra $A$ which we can regard as ``$LT\ltimes_\tau C_0(\ca{M})$.'' Moreover, the pair $(\ca{H},\ca{D})$ determines a smooth algebra $\ca{A}$ of $A$, the triple $(\ca{A},\ca{H},\ca{D})$ is a spectral triple. Lastly, it has an infinite spectral dimension.
\end{thm}

The followings are other evidences.
\begin{cor}
Any irreducible positive energy representation of $LT$ can be obtained from an analytic index of the flag manifold $LT/T$. It is an infinite dimensional version of Borel-Weil's theorem.
\end{cor}
\begin{cor}
If $\ca{M}$ is an $LT$-Hamiltonian space at level $\tau$, it has a quantization valued in $R^\tau(LT)$. It is an infinite dimensional version of Bott's quantization.
\end{cor}
Let us move to the explanations of the details, ideas and techniques of the above.


Let us start from the explanation that Theorem \ref{Thm algebra Sect.1} justifies Theorem \ref{Thm index Sect.1}. For this purpose, we try to rewrite differential geometry as algebraically as possible. Since we can imagine noncommutative version of such a framework, we call it a noncommutative differential geometry here.
In other words, we desire an algebraic setting which knows much information of compact Riemannian manifold $M$. One can consult \cite{Con94} and \cite{Con08}. If the readers are familiar with noncommutative geometry, they may skip to the following of Definition \ref{Dfn n.u.spectral triple Sect.1}.

The answer is as follows;
{\it differential manifold $M$ is reproduced from the $C^*$-algebra $C(M)$, the representation space $L^2(M)$ of $C(M)$, and the elliptic operator $\triangle$ of second order.}
According to well-known Gelfand's theorem, we can reproduce the {\bf topological space} $M$ from the algebra $C(M)$ consisting of continuous functions as the set of maximal ideals. It might remind the readers Hilbert's Nullstellensatz. We are not satisfied it because $C(M)$ does not know smooth structure of $M$ and hence we suppose that we are given a representation of $L^2(M)$ and $\triangle$ which acts on $L^2(M)$. Since $1\in L^2(M)$, $C(M)$ is contained in $L^2(M)$. Considering the domain of $\triangle^{2k}$, we obtain the Sobolev spaces $L^2_{2k}(M)$ and hence $L^2_{\infty}(M)$. Thanks to Sobolev's embedding theorem, $L^2_{\infty}(M)=C^\infty(M)$. Therefore we can reproduce $C^\infty(M)$. According to Parsell-Shanks' theorem, two compact manifolds $M_1$ and $M_2$ are diffeomorphic to one another if and only if $C^\infty(M_1)$ and $C^\infty(M_2)$ are isomorphic to one another as algebras and hence it is an algebraic formulation of differential topology. One can find a proof in \cite{Pra}.

Let us continue the algebraization. We desire more information about $M$. For simplicity, we suppose that $M$ is a $Spin^c$ manifold and $S$ is the Spinor bundle, and we are given a representation space $L^2(M,S)$ and the Dirac operator $D$ acting on it. Since $D$ is elliptic, we can reproduce $C^\infty(M)$. We obtain not only it, but also information about metric of $M$ as follows.
The commutator $[D,f]$ for $f\in C^\infty(M)$ is almost the multiplication by the exterior derivative $df$, and hence $\|{\rm grad}(f)\|=\|[D,f]\|_{{\rm op}}$. 
It tells us that if $f$ satisfies $\|[D,f]\|_{{\rm op}}\leq 1$, $|f(p)-f(q)|$ cannot be greater that $d(p,q)$. More precisely, 
$$d(p,q)=\sup\{|f(p)-f(q)|\mid \|[\ca{D},f]\|\leq1\}.$$
The reader can find the proof in \cite{Con94}.
Therefore, the pair $L^2(M,S)$ and $D$ knows a part of information of Riemannian metric of $M$.

This is how we reach the following noncommutative differential geometry: {\it a $C^*$-algebra $A$ is regarded as a continuous function algebra, a subalgebra $\ca{A}$ of $A$ is regarded as a smooth structure, and a pair $(\ca{H},\ca{D})$ is regarded as a $L^2$-space for Spinor bundle and a Dirac operator}. More precisely we require the following. See \cite{Con94}, \cite{Con08} and \cite{CGRS} for details.
\begin{dfn}
$(1)$ A smooth algebra $\ca{A}$ in $A$ is a $*$-closed Fr\'{e}chet algebra closed under the holomorphic functional calculus.\\
$(2)$ A triple $(\ca{A},\ca{H},\ca{D})$ is called a spectral triple for $A$ if it satisfies the followings:

(a) $\ca{A}$ is smooth in $A$.

(b) $\ca{H}$ is a $*$-representation of $A$.

(c) $\ca{D}$ is a densely defined self-adjoint operator acting on $\ca{H}$.

(d) $\ca{A}$ preserves ${\rm dom}(\ca{D})$.

(e) $[\ca{D},a]$ is bounded for $a\in \ca{A}$.

(f) $(1+\ca{D}^2)^{-1}$ is a compact operator.
\end{dfn}
Needless to say, $\ca{D}$ is regarded as a Dirac operator on a noncommutative object. (e) and (f) are crucial.
(e) means that $\ca{D}$ is ``first order''. (f) means that $\ca{D}$ is ``elliptic''.

The above tells us the following.
Imagine that one try to analyze a certain crazy object, for example a foliation. Suppose that a spectral triple can be defined from the object, and some invariants are obtained from the triple. It is a noncommutative geometrical invariant of that.

Since our manifold is noncompact, we should explain the necessary modification when we replace $M$ with a locally compact manifold $X$. Remember that our manifold is not locally compact and the following is not enough.

The first difficulty is that the spectrum of elliptic operators can be {\bf continuous} not like the case of compact manifold. For example, the spectrum of Laplacian on the real line is $\{\lambda\in\bb{R}\mid \lambda\geq0\}$.

Another difficulty is about function algebras. Although we do not require any additional conditions to define $C(M)$, it has a $C^*$-norm. This is because $M$ is compact and hence in order to move to the case of noncompact manifolds, we must require some additional condition such that the function algebra has a $C^*$-norm. The most natural idea is to consider the algebra $C_b(X)$ consisting of  bounded continuous functions. However, this idea is not good because it is known that $C_b(X)$ is isomorphic to $C(\beta X)$, where $\beta X$ is the Stone-C\v{e}ch compactification of $X$ (\cite{Bla}). Now, we consider the completion algebra $C_0(X)$ of $C_c(X)$ consisting of functions compactly supported. An element of $C_0(X)$ is said to be vanishing at infinity. $C_0(X)$ is a suitable $C^*$-algebra for our purpose.

This is why, we modify the above definition for non-unital $A$.
\begin{dfn}\label{Dfn n.u.spectral triple Sect.1}
$(1)$ A subalgebra $\ca{A}$ in $A$ is a smooth algebra if the unitalization of $\ca{A}$ is smooth in the unitalization of $A$. ``Unitalization'' will be defined in Definition \ref{unitalization}.

$(2)$ A triple $(\ca{A},\ca{H},\ca{D})$ is a spectral triple if it satisfies from (a) to (e) in the above and the following:

(f$'$) $a(1+\ca{D}^2)^{-1}$ is a compact operator for $a\in A$.
\end{dfn}
$a$ in the condition (f$'$) plays a role of cut off function to $supp(a)$ if $a$ is truly compactly supported.

We explain how far our setting is from the above assumption, that is, we explain how bad the ``spectral triple $(C^\infty_0(\ca{M}),L^2(\ca{M},\ca{L}\otimes\ca{S}),\ca{D})$ for the function algebra $C_0(\ca{M})$.''

Firstly, since infinite dimensional space does not have a nice measure, $L^2$-space does not make sense. The readers can understand how bad a measure on an infinite dimensional space is if they recall Bochner-Minlos' theorem in provability theory; a positive definite function on an infinite dimensional vector space $V$ is a characteristic function of a measure on the {\bf algebraic dual} $V^*$.

Secondly, a function on $\ca{M}$ vanishing at infinity is automatically $0$. It can be verified as follows. Suppose that a function $f$ satisfies that $|f(m)|>\varepsilon$ at $m\in\ca{M}$, then there exists an open neighborhood $U$ such that $|f(x)|>\varepsilon/2$ on $U$. However, from the assumption that $f$ vanishes at infinity, $|f(x)|<\varepsilon/2$ in the outside of some compact set. It is a contradiction since $\ca{M}$ is not locally compact. It is a similar situation with foliations. For example, a continuous function on the leaf space of Kronecker's foliation must be a constant. The fact that the classical $C_0(\ca{M})$ must be $\{0\}$ is the reason why we are interested in noncommutative geometry.

Before going further, we explain analytic indices of spectral triples in special cases. Let us notice that spectral triples themselves do not have indices. The index map for a locally compact manifold is a pairing between the set of spectral triples with $K$-group. An element of $K$-group of a locally compact space $X$ is a pair of vector bundles $E_1$ and $E_2$ on X such that they are isomorphic to one another in the outside of some compact set $K$. The pair is regarded as canceled on $X-K$, and hence we can consider the element $E_1-E_2$ to be compactly supported. This situation is parallel to the pairing between Borel-Moore homology and cohomology with compact support. 

Nevertheless, a non-unital spectral triple itself defines an analytic index in some special situation, that is, when a noncompact group acts on a noncompact manifold nicely, and a spectral triple which is compatible with the action defines an equivariant analytic index valued in a $K$-group of a $C^*$-algebra defined by the group acting on the space. Let us explain the result about it due to Kasparov precisely.

Let $X$ be a noncompact $Spin^c$ manifold with a Spinor bundle $S$, $\Gamma$ be a noncompact Lie group, and it acts on $X$ properly and cocompactly. Let us suppose that the action lifts to $S$ and the Dirac operator $D$ is $\Gamma$-equivariant. In this situation, although $D$ is not Fredholm in the classical sense, it is $C^*(\Gamma)$-Fredholm and hence
 $\ker(D)$ and ${\rm coker}(D)$ are finitely generated modules of $C^*(\Gamma)$.  Therefore we can define an analytic index $[\ker(D)]-[{\rm coker}(D)]$ valued in $K_0(C^*(\Gamma))$. 

Our aim is an infinite dimensional version of the above result. In order to approach it, let us focus on more abstract side. It is natural to hope to combine $\Gamma$-equivariant data on $X$ with data on $X/\Gamma$. However, a quotient space can be quite wild. The Borel construction $X\times_\Gamma E\Gamma$ is a modified ``quotient space.'' For example, if the action is free, $X\times_\Gamma E\Gamma$ is homotopy equivalent to $X/\Gamma$. Let us call such an object a ``better quotient.'' The following is a better quotient in the noncommutative world.
\begin{dfn}
When a topological group $\Gamma$ with a Haar measure $\mu$ acts on a $C^*$-algebra $A$ preserving addition, multiplication, norm, scalar multiplication and the involution $*$, we can define a product on $C_c(\Gamma,A)$ by the formula
$$f_1*f_2(\gamma):=\int f_1(\gamma')\gamma'.\Bigl(f_2(\gamma'^{-1}\gamma)\Bigr)d\mu(\gamma').$$
The crossed product $\Gamma\ltimes A$ is defined by a certain completion of $C_c(\Gamma,A)$.
\end{dfn}
Equivariant cohomology $H_\Gamma^*(X):=H^*(X\times_\Gamma E\Gamma)$ controls equivariant data on $X$. Like this, the crossed product $\Gamma\ltimes A$ controls equivariant data on $A$. Using the crossed product, we can rewrite the result by Kasparov.
\begin{pro}\label{Prop Kas Sect.1}
In the situation described above, we can define a spectral triple for $\Gamma\ltimes C_0(X)$ and its analytic index valued in $K_0(C^*(\Gamma))$.
\end{pro}
This is the correct form which we generalize in this paper. 

Before presenting the main result precisely, we must explain representations of $LT$.

Since $LT$ is commutative, any unitary representation is $1$ dimensional. However, it is unbelievable that such representations know some information of analysis or geometry of $LT$. If a representation of $LT$ knows any such information, it will satisfy the followings:
\begin{itemize}
\item it is infinite dimensional.
\item it knows $S^1$ symmetry in some sense.
\item it satisfies some finiteness in some sense.
\end{itemize}
Let us call such a representation a positive energy representation (P.E.R. in short). Needless to say, such a representation cannot exist without any artifice. In fact, a projective representation might satisfy the above condition. A projective representation of $LT$ on a Hilbert space $\ca{H}$ is a map $\rho:LT\to U(\ca{H})$ satisfying that $\rho(l_1)\rho(l_2)=\tau(l_1,l_2)\rho(l_1l_2)$, where $\tau:LT\times LT\to U(1)$ is a group $2$-cocycle, and the cohomology class of it is called level of the representation. Such a $\tau$ determines a central extension of $LT$ and we write the total space of the central extension as $LT^\tau$. $R^\tau(LT)$ denotes the free module generated by isomorphism classes of irreducible P.E.R.s at level $\tau$.

Since $R^\tau(LT)$ has been studied by various people, we seek to define an analytic index valued in $R^\tau(LT)$. In order to do so, we should deal with only $LT^\tau$-equivariant bundle which the added center to $LT^\tau$ acts on by scalar multiplication. Moreover, the crossed product should be twisted, that is, what to define is not ``$LT\ltimes C_0(\ca{M})$'' but ``$LT\ltimes_\tau C_0(\ca{M})$'' considering the cocycle $\tau$. The precise definition of crossed product is in the body.

Let us state the main result precisely. Before that, we notice a canonical decomposition $LT=T\times\Omega T=T\times\Pi_T\times\Omega T_0$, where $\Omega T$ is the based loop group, $\Omega T_0$ is the identity component of it, and $\Pi_T$ is the group consisting of the connected components of $\Omega T$. The decomposition plays a fundamental role in this paper.

\begin{thm}
$(1)$ Let $LT$ act on an infinite dimensional manifold $\ca{M}$ satisfying that:
\begin{itemize}
\item the restriction of the action to $\Omega T$ is free.
\item $\ca{M}/\Omega T$ is a compact $Spin^c$ manifold.
\item $\Omega T_0\to \ca{M} \to \ca{M}/\Omega T_0$ is a trivial bundle.
\end{itemize}
Then, we can define a $C^*$-algebra $LT\ltimes_\tau C_0(\ca{M})$ by using the inductive limit, and $LT\ltimes_\tau C_0(\ca{M})\cong (T\times\Pi_T)\ltimes_\tau C_0(\widetilde{M})\bigotimes \ca{K}(L^2(\Omega T_0))$, where $L^2(\Omega T_0)$ is a Hilbert space obtained by the inductive limit of true $L^2$-spaces of finite dimensional vector spaces, and $\ca{K}(L^2(\Omega T_0))$ is the $C^*$-algebra consisting of all compact operators on $L^2(\Omega T_0)$.

$(2)$ Let $\ca{L}$ be an $LT^\tau$-equivariant line bundle which the added center to $LT^\tau$ acts on by scalar multiplication. We can construct a Spinor bundle $\ca{S}$ on $\ca{M}$, Hilbert space $\ca{H}$ which we regard as ``$L^2(\ca{M},\ca{L}\otimes\ca{S})$,'' and a self-adjoint operator $\ca{D}$ which we regard as a Dirac operator. Moreover, we can define an analytic index $\ind(\ca{D})$ as an element of $R^\tau(LT)$.

$(3)$ We can combine $(1)$ and $(2)$ in terms of spectral triples.
\end{thm}

Let us show a previous study to define a function algebra of a certain infinite dimensional space by using inductive limits in \cite{HKT}.

They considered a sequence of Euclidean spaces $\bb{R}\xrightarrow{i_1}\bb{R}^2\xrightarrow{i_2}\cdots$. In order to define a function $i_{j*}(f)$ for $f\in C_0(\bb{R}^j)$, it is natural to define $i_{j*}(f)$ as a function of the form $f\otimes g$, where $g$ is a function on the orthogonal complement. We desire a ring homomorphism $i_{j*}$ and for it, $g$ must be a projection, that is, $g(x)$ must be $0$ or $1$ at any $x$. It is impossible that $g$ is non zero if we require that $g$ is continuous and vanishes at infinity.

According to their idea, we should give up several properties. Firstly, we should consider not $C_0(\bb{R}^n)$ itself, but an algebra $\mathscr{S}\widehat{\otimes} C_0(\bb{R}^n,\mathit{Cliff}(\bb{R}^n))$, where $\mathscr{S}$ here is $C_0(\bb{R})$ with even/odd grading, and $\mathit{Cliff}(\bb{R}^n)$ is the Clifford algebra. In short, we define the data on the orhogonal completion by $\ca{S}$. More precisely, for $\phi\in C_0(\bb{R}^n,\mathit{Cliff}(\bb{R}^n))$, $f\in \ca{S}$, $x\in \bb{R}^n$ and $y\in \bb{R}$, 
$$i_{n*}(\phi)(x,y):=\phi(x)[f_{{\rm even}}(y)+\text{something}].$$
We must deal with Clifford algebra so that ``something'' in the above makes sense.

Let us return to our case. For the study of $\ca{M}$, it is essential to deal with $\Omega T_0$. It is isomorphic to $C^\infty(S^1,\fra{t})/\fra{t}$ and hence it has a nice basis due to Fourier series theory. By using it, we can define a sequence consisting of sub groups
$$U_1\xrightarrow{i_1} U_2\xrightarrow{i_2}U_3\xrightarrow{i_3}\cdots.$$
Recall that $\ca{M}=\widetilde{M}\times \Omega T_0$. We can define inductive systems of $C^*$-algebras and Hilbert spaces
$$(T\times \Pi_T\times U_1)\ltimes_\tau C_0(\widetilde{M}\times U_1)\to (T\times \Pi_T\times U_2)\ltimes_\tau C_0(\widetilde{M}\times U_2)\to \cdots$$
$$L^2(\widetilde{M}\times U_1)\to L^2(\widetilde{M}\times U_2)\to L^2(\widetilde{M}\times U_3)\to \cdots.$$
$LT\ltimes_\tau C_0(\ca{M})$ and $L^2(\ca{M})$ denote the inductive limit of the aboves respectively. 
In similar ways, we can define $L^2(\ca{M},\ca{L}\otimes\ca{S})$, $L^2(\Omega T_0)$, $L^2(LT)$, $\Omega T_0\ltimes_\tau C_0(\Omega T_0)$, and so on.

We have mentioned better quotients in the noncommutative world in the above. Considering it, our algebra $LT\ltimes_\tau C_0(\ca{M})$ must relate to $(T\times\Pi_T)\ltimes_\tau C_0(\widetilde{M})$, because the normal subgroup $\Omega T_0\triangleleft LT$ acts on $\ca{M}$ freely and the quotient is $\widetilde{M}$.
Our algebra actually satisfies such a property.
\begin{pro}The isomorphism
$$LT\ltimes_\tau C_0(\ca{M})\cong (T\times \Pi_T)\ltimes_\tau C_0(\widetilde{M})\bigotimes \ca{K}(L^2(\Omega T_0))$$
holds. In short, $LT\ltimes_\tau C_0(\ca{M})$ is Morita equivalent to $ (T\times \Pi_T)\ltimes_\tau C_0(\widetilde{M})$.
\end{pro}


This paper is organized as follows.

In Section \ref{Section Preliminaries}, we explain some objects. In order to construct $LT$-equivariant analytic index, we combine noncommutative geometry with representation theory by using a functional analytic technique motivated by a geometric quantization problem, and hence we need many tools and concepts from several areas. Even if the readers are not familiar with all the objects, we believe they understand our work by consulting this section.

In Section \ref{Section simplification}, we present the problem and divide it into two parts. 

In Section \ref{Section Hilbert} and Section \ref{Section Index}, we construct a Hilbert space which we regard as $L^2$-sections of $\tau$-twisted $LT$-equivariant Clifford module bundle, and an operator which we regard a Dirac operator for the bundle. Moreover we define an $LT$-equivariant analytic index valued in $R^\tau(LT)$ for which we study a group $C^*$-algebra of $(T\times\Pi_T)^\tau$. We will define a $C^*$-algebra $C^*(LT,\tau)$ which we will not use though.

In Section \ref{Section Algebra}, we `justify' the package in the previous two sections. More precisely, we construct a $C^*$-algebra which we regard as the twisted crossed product $LT\ltimes_\tau C_0(\ca{M})$. 
Moreover, we combine all of them in terms of spectral triples. 

In Section \ref{Section Applications}, we present some applications.

\section{Preliminaries}\label{Section Preliminaries}
Since we hope so many people to read this paper, we collect definitions and explanations of objects which we need, for convenience of the readers. Firstly, we explain loop groups, for which we must start from representation theory of compact Lie groups. Secondly, we explain $C^*$-algebras. Although we have explain noncommutative geometry we focus on in the previous section, we need deeper results of $C^*$-algebras. Lastly, we briefly explain geometric quantization program which is our original motivation.


\subsection{Compact Lie groups}
We learned this subject from \cite{KO} in Japanese.

Firstly, we list properties of representations of compact Lie groups.
Before that, we recall representation theory of a torus. We use $S$ for a torus here.
Let us notice that a torus $S$ determines two discrete groups
\begin{eqnarray*}
\Pi_S &:=& \ker(\exp:\fra{s}\to S)\subseteq\fra{s} \\
\Lambda_S &:=& \Hom(\Pi_S,\bb{Z})\subseteq\fra{s}^*.
\end{eqnarray*}
They often appear in this paper.
\begin{pro}
Each irreducible representation of $S$ is $1$-dimensional and unitarizable, and hence the group of isomorphism classes of $S$ under the tensor product is isomorphic to $\Hom(S,U(1))$. 
We can define a representation of $\fra{s}$ and hence an element of $\fra{s}^*$ from a representation of $S$.
Considering the infinitesimal representations, $\Hom(S,U(1))$ is isomorphic to $\Hom(\fra{s},\bb{Z})$.
\end{pro}

Let $G$ be a compact, semisimple, connected and simply connected Lie group.
It is known that $G$ has a maximal connected commutative sub group, so called a maximal torus, and all maximal tori are conjugate to one another. Let $S$ be a chosen maximal torus. It determines its normalizer $S\triangleleft N(S)\subseteq G$ and the Weyl group $W(G):=N(S)/S$. If we choose another maximal torus $S'$, the induced Weyl group $N(S')/S'$ is isomorphic to the original one, and hence the Weyl group is independent of the choice of a maximal torus.
\begin{pro}
$(1)$ All representations on Hilbert spaces are unitarizable, that is, for any representation $(\ca{H},\rho)$, there exists maybe a new inner product, $\rho$ is unitary under it.

$(2)$ Each irreducible representation is finite dimensional.

$(3)$ The restriction of each irreducible representation to the maximal torus $S$ determines a weighted subset of $\Lambda_S$. It is $W(G)$-invariant. Moreover, if two irreducible representations $\rho_1$ and $\rho_2$ are not isomorphic to one another, the subsets are different. It is known as the Cartan-Weyl theory. $\widehat{G}$ denotes the set of isomorphism classes of irreducible representations.
\end{pro}
If we take a basis in $\Lambda_S$, the lexicographical order $<$ is defined. If we are given two elements $\lambda_1$ and $\lambda_2$ such that $\lambda_1<\lambda_2$, $\lambda_2$ is said to be higher than $\lambda_1$. The set constructed in $(3)$ is completely determined by the highest weight.
The following is well-known Borel-Weil theory. We state $LT$ version of it as an application of the main result.
\begin{pro}
Suppose that we are given an irreducible representation $(V,\rho)$. It can be realized as the space of holomorphic sections of a line bundle obtained by the highest weight $\lambda$. 
\end{pro}

Our ``$L^2$-space'' will be constructed modeled on Peter-Weyl theorem. Thus we explain it more precisely. It states that we can use representation theory to analyze $L^2$-theory of compact groups.

Let $G$ be a compact group and hence there exists the unique biinvariant Borel measure $\mu$ on it such that $\mu(G)=1$, which is called the normalized Haar measure.
We can define two canonical representations $L$ and $R$, so called the left and right regular representations respectively.
\begin{dfn}
Let $\phi\in L^2(G)$ and $g,x\in G$. $L(g),R(g):L^2(G)\to L^2(G)$ is defined by
$$\bigl[L(g)\phi\bigr](x):=\phi(g^{-1}x)$$
$$\bigl[R(g)\phi\bigr](x):=\phi(xg).$$
\end{dfn}
It is a unitary representation since $\mu$ is biinvariant. We wish to decompose it as a direct sum of irreducible representations. Before that, we explain how a function can be obtained by a representation.
 Let $(V,\rho)$ be an irreducible representation. $f\otimes v \in V^*\otimes V$ gives a function
$$[f\otimes v](x):=\innpro{f}{\rho(g^{-1})v}{}.$$
Let us call this correspondence a matrix coefficient and $\Psi_\rho:V^*\otimes V\to L^2(G)$ denotes the map.
However, it can be not isometric unless we multiply normalization factors.
\begin{thm}[Peter-Weyl]
For each irreducible representation $(V_\rho,\rho)$, the normalized matrix coefficients $\sqrt{\dim(V_\rho)}\Psi_\rho$ are isometric. Moreover the direct sum of them
$$\bigoplus_{\rho\in \widehat{G}}\sqrt{\dim(V_\rho)}\Psi_\rho:\bigoplus_{\rho\in \widehat{G}}V_\rho^*\otimes V_\rho\to L^2(G)$$ 
is an isometric isomorphism as representation spaces if we take the completion.
\end{thm}
Thanks to the above, representation theory can be a quite powerful tool for harmonic analysis of compact Lie groups. For example, the infinitesimal versions of $R$ and $L$ correspond to direction derivatives in the Peter-Weyl theorem and hence representation theory knows all invariant differential operators. In particular, the Casimir operator $\triangle:=\sum dR(e_i)^2$ corresponds to the Laplacian, where $\{e_i\}$ is an orthonomal basis of $\fra{g}$ with respect to an invariant inner product. 

In a similar idea, we can define a Dirac operator in terms of representation theory.
One can consult \cite{Kos}.
Let $(S(\fra{g}),\gamma)$ be a Spinor of $\fra{g}$ with respect to an invariant inner product. 
\begin{dfn}
An algebraic Dirac operator $D$ is defined by
$$D:= \sum dR(e_i)\otimes\gamma(e_i)$$
and it acts on $L^2(G,G\times S(\fra{g}))=L^2(G)\otimes S(\fra{g})$.
\end{dfn}
\begin{rmk}
We omit the cubic term for simplicity, and thus, our $D$ is not FHT's unless $G$ is commutative.
\end{rmk}

Let us observe it and understand how algebraic it is. Peter-Weyl theorem enables us to consider $v\otimes f\otimes s\in V_\rho\otimes V^*_\rho\otimes S(\fra{g})$ as an $L^2$-section of $G\times S(\fra{g})$, where we switch $V_\rho$ and $V_\rho^*$ for  simplicity. The action of $D$ on it can be computed as follows;
\begin{eqnarray*}
D(v\otimes f\otimes s)&=&\sum dR(e_i)(v\otimes f)\gamma(e_i)s \\
&=& \sum v\otimes d\rho^*(e_i)f\otimes\gamma(e_i)s \\
&=& v\otimes\Bigl(\sum d\rho^*(e_i)f\otimes \gamma(e_i)s\Bigr).
\end{eqnarray*}
If we set $D_\rho:=\sum d\rho^*(e_i)\otimes\gamma(e_i)$, it is an operator acting on a {\bf finite dimensional} space and the Dirac operator can be written as the direct sum of $D_\rho$'s. This is why they call their operator ``Dirac operator'' although their Hilbert space is just a tensor product of representation space with Spinor space.

\subsection{Loop groups and its representation}
Our main references of this subsection are \cite{PS} and \cite{FHTII}.
The loop group $LG$ for a compact Lie group $G$ is the space of $C^\infty$ maps from a circle $\bb{T}_{{\rm rot}}$ to $G$. Since this paper is full of circles, we use different symbols to distinguish them.
The circle for the target of the loop group is $T$, and the circle for central extensions is $U(1)$.
Although $LG$ has a big symmetry $\mathit{Diff}(\bb{T}_\rot)$, the diffeomorphism group of $\bb{T}_\rot$, we use only the symmetry coming from the rigid rotations defined below.
\begin{dfn}
For $\theta_0\in\bb{T}_\rot$ and $l\in LG$,
$$\theta_0.l(\theta):=l(\theta+\theta_0).$$
The infinitesimal generator of it is written as $d:L\fra{g}\to L\fra{g}$.
\end{dfn}
We hope to deal with $LG$ by similar techniques of compact Lie groups. For this aim, we desire a nice class of representations as we explained in Section \ref{Section Introduction}
Let us present the precise definition of P.E.R.s.
Before that, we need to explain what kind of central extension is suitable.
\begin{dfn}[\cite{FHTII}{Definition 2.10.}]
A central extension $1\to U(1) \xrightarrow{i} LG^\tau \xrightarrow{p} LG\to 1$ is admissible if:

$(1)$ The rigid rotation action on $LG$ lifts to $LG^\tau$, therefore, we can construct the semi direct product $\bb{T}_{{\rm rot}}\ltimes LG^\tau$.

$(2)$ There exists $\bb{T}_{{\rm rot}}\ltimes LG^\tau$-invariant symmetric bilinear form $\Innpro{\cdot}{\cdot}{\tau}$ on its Lie algebra such that
$$\Innpro{K}{d}{\tau}=-1$$
for $d\in {\rm Lie}(\bb{T}_{{\rm rot}}\ltimes LG^\tau)$ which project to the generator in ${\rm Lie}(\bb{T}_{{\rm rot}})$ and $K$ is the generator of the added circle $i(U(1))\subseteq LG^\tau$.
\end{dfn}
We are in the position to define P.E.R.s.
\begin{dfn}
A strongly continuous unitary representation $\rho:LG^\tau\to PU(\ca{H})$ is said to be a P.E.R. at level $\tau$ if

$(1)$ The added center $i(U(1))$ acts on $\ca{H}$ by the scaler multiplication.

$(2)$ $\rho$ lifts to the semi direct product $\bb{T}_\rot\ltimes LG^\tau$.

$(3)$ When we decompose $\ca{H}$ via the weight of the action of $\bb{T}_\rot$, as $\ca{H}=\oplus_n\ca{H}_n$, each $\ca{H}_n$ is finite dimensional and $\ca{H}_n=0$ for sufficiently small $n$.
\end{dfn}
\begin{dfn}
The set of isomorphism classes of P.E.R.s at level $\tau$ consists of semigroup under the direct sum. Its Grothendieck completion is denoted by $R^\tau(LT)$.
\end{dfn}

When $G$ is a torus $T$, we can avoid use of central extensions.

Let us recall the canonical decomposition mentioned above, $LT=T\times \Omega T=T\times\Pi_T\times \Omega T_0$. This decomposition implies the homotopy equivalence $LT\approx T\times\Pi_T$ and hence any $U(1)$-principal bundle is trivial. Therefore, there exists a global continuous section $s:LT\to LT^\tau$. It determines a $2$-cocycle
$$\tau(l_1,l_2):=s(l_1)s(l_2)s(l_1l_2)^{-1}\in\ker(p)=i(U(1))$$
where we use the same symbol for the $2$-cocycle. If we take another section $s'=s\cdot f$, where $f:LT\to U(1)$ is a continuous $U(1)$ valued function, the new cocycle is as follows;
\begin{eqnarray*}
\tau'(l_1,l_2) &=& s'(l_1)s'(l_2)s'(l_1l_2)^{-1} \\
&=& s(l_1)s(l_2)s(l_1l_2)^{-1}f(l_1)f(l_2)f(l_1l_2)^{-1} \\
&=& \tau(l_1,l_2)\delta(f)(l_1,l_2).
\end{eqnarray*}
$\delta f$ is the group coboundary of $f$.
It tells that the cohomology class does not depend on the choice of $s$. We define P.E.R.s again. $\tau$  in the following denotes a chosen cocycle.
\begin{dfn}
A continuous map $\rho:LT\to U(\ca{H})$ is a P.E.R. at level $\tau$ if

$(1)$ $\rho(l_1)\rho(l_2)=\tau(l_1,l_2)\rho(l_1l_2)$.

$(2)$ $\rho$ lifts to $\bb{T}_\rot\ltimes LT$.

$(3)$ When we decompose $\ca{H}$ via the weight of the action of $\bb{T}_\rot$, as $\ca{H}=\oplus_n\ca{H}_n$, each $\ca{H}_n$ is finite dimensional and $\ca{H}_n=0$ for sufficiently small $n$.
\end{dfn}
In this paper, we use this model to study P.E.R.s. 

Let us review the classification of cocycles coming from admissible central extensions, and P.E.R.s at each level following Proposition 2.27 in \cite{FHTII}.

\begin{pro}\label{Prop cocycle Sect.2}
For any $2$-cocycle $\tau$ which comes from an admissible central extension,
there exists a group homomorphism $\kappa^\tau:\Pi_T\to \Lambda_T$ such that:
\begin{itemize}
\item $\tau\bigl(l,(t,n)\bigr)=\tau\bigl((t,n),l\bigr)=1$
\item $\tau\bigl((t_1,n_1),(t_2,n_2)\bigr)=\kappa^\tau(n_1)(t_2)$
\item $\tau\bigl(l_1,l_2\bigr)=\exp\bigl( i\Innpro{l_1}{dl_2/d\theta}{\tau}\bigr)$
\end{itemize}
for any $l,l_1,l_2\in \Omega T_0$, $t,t_1,t_2\in T$ and $n,n_1,n_2\in\Pi_T$.
In short, $\Omega T_0$ and $T\times \Pi_T$ commutes with one another, the cocycle on $\Omega T_0$ comes from the canonical symplectic form, and the commutator in $T\times \Pi_T$ defines a homomorphism.
\end{pro}

Let us construct such a cocycle.
Take a homomorphism $\kappa^\tau:\Pi_T\to \Lambda_T$ such that the induced bilinear form on $\fra{t}$ is positive definite. $\innpro{v}{w}{\tau}:=\kappa^\tau(v)(w)$.
The bilinear form required in the definition of admissibility is defined by $\Innpro{v}{w}{\tau}:=\int_{\bb{T}_\rot}\innpro{v}{w}{\tau}d\theta$.
These data determine a $2$-cocycle $\tau$ as in the above Proposition.
We set a symplectic form $\omega(v,w):=\int \innpro{v}{dw/d\theta}{\tau}d\theta$.

For this $\tau$, we can construct P.E.R.s at level $\tau$. Firstly, $\Omega T_0$ has the unique irreducible P.E.R. $V_\ca{H}$ which we will construct explicitly in Section \ref{Section Hilbert}. 
We can make $T\times\Omega T_0$ act on $V_\ca{H}\otimes\bb{C}_\lambda$ via the tensor product of the actions. $T$ acts on $\bb{C}_\lambda$ via the character $\lambda$. We can not define an action on $\Pi_T$ because of the relation $\tau\bigl((t_1,n_1),(t_2,n_2)\bigr)=\kappa^\tau(n_1)(t_2)$. To define a representation, we need to take an infinite sum 
$$\bigoplus_{n\in\Pi_T}V_\ca{H}\otimes\bb{C}_{\lambda+\kappa^\tau(n)},$$
where $\Pi_T$ permits the component and $T\times\Omega T_0$ acts on each component naturally. Each of irreducible P.E.R.s is of the above form. Therefore, isomorphism classes of irreducible P.E.R.s are in 1:1 correspondence with the orbit space $\Lambda_T/\kappa^\tau(\Pi_T)$.
\begin{pro}\label{Prop RLT Sect.2}
$$R^\tau(LT)\cong\bigoplus_{[\lambda]\in\Lambda_T/\kappa^\tau(\Pi_T)}\bb{Z}.$$
\end{pro}

\begin{rmk}$\tau^{-1}$ denotes the cocycle coming from $-\kappa^\tau$. For example, the dual representation of a P.E.R. is a negative energy representation which the readers might guess the definition at level $\tau^{-1}$.
\end{rmk}

\subsection{$C^*$-algebras}\label{Subsection Cstar algebra}

We review a very beautiful theory of $C^*$-algebras and some topics we need very quickly. If the readers are interested in the topics, they can consult \cite{Mur}.
\begin{dfn}
An algebra $A$ over $\bb{C}$ is a $C^*$-algebra if
\begin{itemize}
\item it has an anti-linear involution $*$,
\item it has a complete norm $\|\cdot\|$,
\item $\|ab\|\leq \|a\|\|b\|$,
\item $\|a^*a\|=\|a\|^2$.
\end{itemize}
The last equality is called the {\it $C^*$-condition}. A $C^*$-algebra with unit is called a {\it unital $C^*$-algebra} and if not, a {\it non-unital $C^*$-algebra}.

A ring homomorphism $f:A\to B$ between $C^*$-algebras is a $*$-homomorphism if it preserves the involution $*$.
\end{dfn}
\begin{exs}
$(1)$ When $\ca{H}$ is a Hilbert space, the algebra $\ca{B}(\ca{H})$ consisting of bounded operators is a $C^*$-algebra with respect to the operator norm. The only closed ideal $\ca{K}(\ca{H})$ is also a $C^*$-algebra with respect to the induced norm. $\ca{K}(\ca{H})$ is the most fundamental non-unital $C^*$-algebra and sometimes appears in this paper.

$(2)$ For a compact Hausdorff space $X$, the algebra $C(X)$ consisting of continuous functions is a commutative $C^*$-algebra with respect to the sup norm.
\end{exs}

A representation of a group on a vector space $V$ is the group homomorphism from a group to the unitary group of $V$. Like this,
$*$-homomorphism $A\to \ca{B}(\ca{H})$ is called a {\it $*$-representation} on $\ca{H}$.

Eigenvalues of a matrix play important roles to analyze the matrix. Spectrum of bounded operators is an infinite dimensional version of the set of eigenvalues. We can generalize this concept to general $C^*$-algebras. Before that, we need to introduce the unitalization of $C^*$-algebras.
\begin{dfn}\label{unitalization}
$(1)$ Let $A$ be a non-unital $C^*$-algebra. $\widetilde{A}:=A\oplus \bb{C}$ is a $*$-algebra by the following involution and multiplication
$$(a,\lambda)^*:=(a^*,\overline{\lambda})$$
$$(a,\lambda)\cdot(a',\lambda'):=(aa'+\lambda a'+\lambda'a,\lambda\lambda').$$
$\widetilde{A}$ acts on $A$ by the action $(a,\lambda)\cdot b:=ab+\lambda b$ and hence we have a homomorphism $\pi:\widetilde{A}\to \ca{B}(A)$, where $\ca{B}(A)$ is the algebra consisting of bounded operators of $A$ regarding $A$ as a Banach space. If we define a norm on $\widetilde{A}$ by the norm of $\pi(a)$, it satisfies the $C^*$-condition. We often use such a construction of norms in several situations.

$(2)$ Let $A$ be a unital $C^*$-algebra, $a\in A$. The resolvent set $\rho(a)$ of $a$ is the set of $\lambda\in\bb{C}$ such that $a-\lambda\cdot 1_A$ is invertible in $A$. The spectrum is the complement of $\rho(a)$ and it is written as $\sigma(a)$.

$(3)$ For a non-unital $C^*$-algebra, the resolvent and the spectrum in $A$ is defined by ones in $\widetilde{A}$.
\end{dfn}
It is known that $\rho(a)$ is a compact subset in $\bb{C}$ just like ordinary spectrum of bounded operators.
Thanks to spectral theory, the followings hold.
\begin{pro}
$(1)$ A $*$-homomorphism is automatically norm decreasing.

$(2)$ An injective $*$-homomorphism is automatically isometric.
\end{pro}
All such unbelievably beautiful results are essentially due to the $C^*$-condition which can be considered as a compatibility condition between algebra and analysis.

We need some additional concepts, tensor products and inductive limits.
They are not straightforward. Recall the case of Hilbert spaces; the algebraic tensor product of two Hilbert spaces is not complete and hence we must take a completion. In the case of $C^*$-algebras, the situation is much more complicated. 

\begin{dfn}
Let $A$ and $B$ be $C^*$-algebras. $A\otimes^\alg B$ is the algebraic tensor product over $\bb{C}$. Although it is an algebra equipped with an involution, it is not complete in general and we must take a completion to define a tensor product as a $C^*$-algebra. We introduce two norms by the formula
\begin{eqnarray*}
&&\|\sum a_i\otimes b_i\|_{{\rm min}}:= \\
&&\;\;\;\sup\{\|\sum\pi_A(a_i)\otimes \pi_B(b_i)\|_\op\mid\pi_A,\pi_B\text{: representations of }A,B\text{ respectively}\} \\
&&\|\sum a_i\otimes b_i\|_{{\rm Max}}:= \\
&&\;\;\;\sup\{\|\sum \pi(a_i\otimes b_i)\|_\op\mid\pi\text{: a representation of the algebra }A\otimes^\alg B\}.
\end{eqnarray*}
$A\otimes_{{\rm Max}} B$ and $A\otimes_{{\rm min}} B$ are the completion of $A\otimes^\alg B$ with respect to $\|\cdot\|_{{\rm Max}}$ and $\|\cdot\|_{{\rm min}}$ respectively. These algebras are called the maximal/minimal tensor product algebra respectively.
\end{dfn}
In general, two tensor products are not the same. However, we do not face such a difficulty as follows.
\begin{pro}
If either $A$ or $B$ is commutative, $A\otimes_{{\rm min}}B=A\otimes_{{\rm Max}}B$. Like this, if either $A$ or $B$ is the $C^*$-algebra $\ca{K}(\ca{H})$ consisting of compact operators on a Hilbert space, $A\otimes_{{\rm min}}B=A\otimes_{{\rm Max}}B$. More generally, a $C^*$-algebra $D$ is said to be nuclear if $D\otimes_{{\rm min}}A=D\otimes_{{\rm Max}}A$ for any $C^*$-algebra $A$.
\end{pro}
We are always in the above situation in this paper and hence we can omit the subscript ${{\rm Max}}$ and ${{\rm min}}$. However we might add the subscript when we emphasize the way to complete.

We show some examples of tensor products.
\begin{exs}
$(1)$ Let $X$ and $Y$ be compact Hausdorff spaces. Then, $C(X)\otimes C(Y)\cong C(X\times Y)$ due to Stone-Weierstrass theorem.

$(2)$ Let $\ca{H}_1$ and $\ca{H}_2$ be Hilbert spaces. Then $\ca{K}(\ca{H}_1)\otimes \ca{K}(\ca{H}_1)\cong \ca{K}(\ca{H}_1\otimes\ca{H}_2)$.
\end{exs}

Let us explain inductive limits. We can formulate the concept of inductive limits in any category, that is, the inductive limit of the system
$$A_1\xrightarrow{i_1}A_2\xrightarrow{i_2}A_3\xrightarrow{i_3}\cdots$$
is an object $A$ equipped with $j_n:A_n\to A$ such that for any $B$ equipped with $k_n:A_n\to B$ satisfying that $k_n=k_{n+1}\circ i_n$, there uniquely exists $f:A\to B$ satisfying that $f\circ j_n=k_n$. We can easily show a category such that it is not closed under the inductive limit. For example, the category consisting of finite dimensional vector spaces is such a category. However, the category consisting of $C^*$-algebras is not.
\begin{pro}
For any sequence of $C^*$-algebra
$$A_1\xrightarrow{i_1}A_2\xrightarrow{i_2}A_3\xrightarrow{i_3}\cdots,$$
where $i_j$'s are $*$-homomorphisms, there exists a $C^*$-algebra $A$ such that $A$ is the inductive limit of the above sequence. The inductive limit is written as $\varinjlim A_n$.
\end{pro}
\begin{ex}
Let $i_n:M_n(\bb{C})\to M_{n+1}(\bb{C})$ be a $*$-homomorphism defined by
$$a\mapsto \begin{pmatrix} a & 0 \\ 0 & 0\end{pmatrix}.$$
The inductive limit of this sequence is $\ca{K}(\ca{H})$.
\end{ex}
In short, the inductive limit of $C^*$-algebra is the completion of the algebraic limit.

Let us move to noncommutative geometry. At the risk of being misunderstood, noncommutative geometry is a study of invariant of $C^*$-algebras. Let us justify this crazy idea. We have explained that compact Hausdorff space determines a commutative $C^*$-algebra. The converse direction can be also established.

\begin{pro}
The correspondence
\begin{center}
$X$ $\mapsto$ $C_0(X)$
\end{center}
gives an equivalence between a category of locally compact Hausdorff topological spaces and a one of commutative $C^*$-algebras.
\end{pro}
Generalizing it, we regard a $C^*$-algebra as a function algebra vanishing at infinity of a locally compact ``noncommutative space'' or ``quantum space.'' 

A fist obvious example of noncommutative $C^*$-algebra is a matrix algebra $M_n(\bb{C})$ or the algebra consisting of compact operators on a Hilbert space $\ca{K}(\ca{H})$. However, we regard them as a strange function algebra on a point because of the following.
\begin{pro}\label{Prop Morita Sect 2}
$M_n(\bb{C})$ and $\ca{K}(\ca{H})$ are Morita equivalent to $\bb{C}=C(\{{\rm pt}\})$. More precisely, any finitely generated module of $M_n(\bb{C})$ is of the form $(\bb{C}^n)^m$ and one of $\ca{K}(\ca{H})$ is of the form $\ca{H}^m$.
\end{pro}
It is a typical example of Morita equivalence. In short, two algebras are Morita equivalent to one another if the categories of modules are equivalent. However, to state the definition precisely, we must prepare some new notions including Hilbert $C^*$-modules which we do not need. This is why we state the equivalent condition which is easier to understand than the original definition.
\begin{pro}[\cite{Lan}]
Two $C^*$-algebras $A$ and $B$ satisfying a certain mild condition are Morita equivalent if and only if $A\otimes\ca{K}(\ca{H})\cong B\otimes\ca{K}(\ca{H})$.
\end{pro}
It relates with local equivalence or Morita equivalence of groupoids. For details, one can consult \cite{Mei}.

Rewriting of classical geometrical objects is a start point of noncommutative geometry. 
We have mentioned spectral triples which is a noncommutative version of Dirac operators in Section \ref{Section Introduction}. We add another one here. The following is a noncommutative version of a better quotient space by a group action on a topological space. We refer to \cite{Wil} and \cite{Kha}.

Let $\Gamma$ be a locally compact unimodular group, and act on a $C^*$-algebra $A$, that is, we are given a group homomorphism $\sigma:\Gamma\to {\rm Aut}(A)$. The convolution product in $C_c(\Gamma,A)$ is defined by the formula
$$f_1*f_2(\gamma):=\int_\Gamma f_1(\gamma')\sigma(\gamma')\Bigl(f_2(\gamma'^{-1}\gamma)\Bigr)d\gamma'.$$
It has an involution defined by
$$f^*(g):=\sigma(g)\bigl(f(g^{-1})^*\bigr).$$
We take completions by some norms like the case of tensor products. We show two norms, $\|\cdot\|_{{\rm Max}}$ and $\|\cdot\|_{{\rm red}}$. 

Before that, we need to introduce a terminology ``covariant representation.''
A representation $(\pi,\ca{H})$ of $A$ equipped with a linear action of $\rho:\Gamma\to U(\ca{H})$ is said to be a covariant representation if it satisfies that
$$\rho_\gamma(\pi_a(v))=\pi_{\gamma.a}(\rho_\gamma(v)).$$
Omitting the name of maps, $\gamma.(av)=(\gamma.a).(\gamma.v)$. We call such a pair $(\pi,\rho)$ a covariant representation.
It defines a representation of the algebra $C_c(\Gamma,A)$ by the formula
$$\pi\rtimes\rho_f(v):=\int_\Gamma \pi_{f(\gamma)}(\rho_\gamma(v))d\mu(\gamma).$$

As an example, we present a canonical covariant representation $(L,\widetilde{\pi})$ on $L^2(\Gamma,\ca{H})$ induced by a faithful representation $(\ca{H},\pi)$ of $A$ by the formula $\widetilde{\pi}_a(f)(x):=\pi(x^{-1}.a)f(x)$. $L$ is the left regular representation $L_\gamma(f)(x):=f(\gamma^{-1}x)$. Let us call this covariant representation a canonical covariant representation. Let us notice that any $C^*$-algebra has a faithful representation thanks to Gelfand-Neimark-Segal theorem (\cite{Mur}).

Let us define two norms like the case in tensor products.
\begin{dfn}\label{Dfn norms Sect.2}
Let $f\in C_c(\Gamma,A)$. 
\begin{eqnarray*}
\|f\|_{{\rm Max}}&:=&\sup\{\|\pi\rtimes\rho_f\|_\op\mid(\pi,\rho)\text{: a covariant representation}\} \\
\|f\|_{{\rm red}}&:=&\|\widetilde{\pi}\rtimes L_f\|_\op \text{ for a faithful }\pi.
\end{eqnarray*}
The completion of $C_c(\Gamma,A)$ with respect to $\|\cdot\|_{{\rm Max}}$ and $\|\cdot\|_{{\rm red}}$ is written as $\Gamma\ltimes_{{\rm Max}} A$ and $\Gamma\ltimes_{{\rm red}} A$, and is called max/reduced crossed product respectively. Since a faithful representation of $C^*$-algebras is isometric, the definition of $\|\cdot\|_{{\rm red}}$ is independent of the choice of $\pi$.
\end{dfn}
\begin{rmks}
$(1)$ It is not always true that $\Gamma\ltimes_{{\rm Max}} A=\Gamma\ltimes_{{\rm red}}A$. In fact, what we can verify generally is only the existence of a projection $\Gamma\ltimes_{{\rm Max}} A\to\Gamma\ltimes_{{\rm red}}A$. 

$(2)$ Two algebras $\Gamma\ltimes_{{\rm Max}} A$ and $\Gamma\ltimes_{{\rm red}}A$ coincide if $\Gamma$ is amenable (Theorem 7.13 in \cite{Wil}). This fact enables us to benefit from the properties of the both crossed products. In such a case, we omit the subscript.

$(3)$ In short, a locally compact group is amenable if it has an invariant integral for essentially bounded functions with respect to the Haar measure. We quote some properties of amenability from \cite{Pie}:
\begin{itemize}
\item a group extension of an amenable group by an amenable group is again amenable. In particular, a direct product of two amenable groups is again amenable.
\item a compact group is amenable.
\item a locally compact abelian group is amenable.
\end{itemize}
\end{rmks}


Let us show some examples which tells us that a crossed product can be considered as a better quotient. For details, consult \cite{Kha}.
\begin{exs}
$(1)$ When $\Gamma$ acts on itself via the left translation, $\Gamma\ltimes C_0(\Gamma)$ is isomorphic to $\ca{K}(L^2(\Gamma))$ and it is Morita equivalent to $\bb{C}$.

$(2)$ More generally, if $\Gamma$ acts on a locally compact Hausdorff space $X$ freely and properly, $\Gamma\ltimes C_0(X)$ is Morita equivalent to $C_0(X/\Gamma)$.

$(3)$ When the homomorphism $\sigma$ is trivial, $\Gamma\ltimes_{{\rm red}} A\cong C^*_{\rm red}\otimes_{{\rm min}}A$, and $\Gamma\ltimes_{{\rm Max}} A\cong C_{{\rm Max}}^*(\Gamma)\otimes_{{\rm Max}}A$. $C^*(\Gamma)$ is defined below.
\end{exs}

A special case when $A=\bb{C}$ and the action is trivial, of crossed product $\Gamma\ltimes\bb{C}$ is called a full group $C^*$-algebra of $\Gamma$ and written as $C^*(\Gamma)$. Like this, $C^*_{{\rm red}}(\Gamma):=\Gamma\ltimes_{{\rm red}}\bb{C}$.
We regard a module of $C^*(\Gamma)$ as a representation with suitable size. We present some examples of group $C^*$-algebras.
\begin{exs}
$(1)$ When $\Gamma$ is a compact group, 
$$C^*_{{\rm red}}(\Gamma)\cong C^*(\Gamma)\cong\sum_{\lambda\in\widehat{\Gamma}} {\rm End}(V_\lambda).$$

$(2)$ More generally, if $\Gamma$ is amenable, $C^*_{{\rm red}}(\Gamma)\cong C^*(\Gamma)$.

$(3)$ Let $F_2$ be the free group generated by two elements. It is known that $C^*(F_2)$ is not isomorphic to $C^*_{{\rm red}}(F_2)$.
\end{exs}

We are in the position to explain twisted version of crossed products and twisted group $C^*$-algebras. Let $\Gamma$ and $A$ be as the above. Suppose that we are given a $2$-cocycle $\sigma:\Gamma\times\Gamma\to U(1)$.
\begin{dfn}
The twisted convolution product in $C_c(\Gamma,A)$ is defined by
$$f_1*_\sigma f_2(\gamma):=\int_\Gamma f_1(\gamma')\gamma.\bigl(f_2(\gamma'^{-1}\gamma)\bigr)\sigma(\gamma',\gamma)d\mu(\gamma').$$
The twisted crossed product $\Gamma\ltimes_{\sigma,{\rm Max}} A$ and $\Gamma\ltimes_{\sigma,{\rm red}}A$ are the completions of it just like the ordinary crossed products.
\end{dfn}
An $A$-module $M$ is a $\sigma$-twisted $\Gamma$-$A$-bimodule if it admits a map $\rho:\Gamma\to {\rm Aut}(A)$ satisfying that $\rho_{\gamma_1}\rho_{\gamma_2}=\sigma(\gamma_1,\gamma_2)\rho_{\gamma_1\gamma_2}$, such that
$$\rho_\gamma(am)=(\gamma.a)(\rho_\gamma(m)).$$
A twisted crossed product controls $\sigma$-twisted $\Gamma$-$A$-bimodule just like the case of ordinary ones.


Let us notice that we can realize 
the algebra $C_c(\Gamma,A)$ with the above product
as a subalgebra of a certain ordinary crossed product as follows.

The cocycle $\sigma$ defines a group structure on $\Gamma^\sigma:=\Gamma\times U(1)$ by the formula $(\gamma_1,z_1)\cdot(\gamma_2,z_2)=(\gamma_1\gamma_2,z_1z_2\sigma(\gamma_1,\gamma_2))$. Let $C_c(\Gamma^\sigma,A)(k)$ be the subalgebra consisting of functions such that $f(\gamma,z)=f(\gamma,1)z^{-k}$. $C_c(\Gamma^\sigma,A)(1)$ is an algebra equipped with the above twisted convolution. Twisted group $C^*$-algebras can be obtained by a similar method. 
We will explain such things in Section \ref{Section Index}.

Such a perspective enables us to translate results of ordinary crossed products into twisted ones without much effort. For example, $\Gamma\ltimes_{\tau,{\rm Max}}A=\Gamma\ltimes_{\tau,{\rm min}}A$ for amenable $\Gamma$.

\subsection{Geometric quantization}
Let us explain a quantization procedure along \cite{Son}, \cite{Sja} and \cite{AMM}, by which we are motivated.

While classical  mechanics can be considered as a dynamical system on a symplectic manifold, quantum mechanics is a unitary representation on a Hilbert space which is called a state space. We wish to obtain a quantum system from a classical system. Geometric quantization is a program seeking such correspondence.

If a classical system has a symmetry, we hope that corresponding quantum system inherits this symmetry. Bott's quantization is such a quantization.
\begin{dfn}
Let $(M,\omega)$ be a Hamiltonian $G$-space, that is, $(M,\omega)$ is a symplectic manifold, $G$ acts on $M$ preserving $\omega$, and there exists a map $\mu:M\to \fra{g}^*$ so called a momentum map satisfying that $d(\mu(m),\xi)=\iota_{\xi_M}\omega(m)$, where $\xi_M$ is the induced vector field by $\xi\in\fra{g}$. We suppose that the symplectic structure admits a $G$-invariant almost complex structure which is compatible with the symplectic structure and hence a Spinor bundle $S=S^+\oplus S^-\to M$.

The $G$-equivariant line bundle $L\to M$ with connection $\nabla$ is called a pre-quantum line bundle if $\nabla^2=i\omega$. A Dirac operator $D_L:L^2(M,S^+\otimes L)\to L^2(M,S^-\otimes L)$ determines a quantization
$$Q(M,\omega):=[\ker(D_L)]-[{\rm coker}(D_L)]\in R(G).$$

\end{dfn}

In \cite{Son}, he generalized it to the case of Hamiltonian $LG$-spaces.
\begin{dfn}
A Hamiltonian $LG$-space is an infinite dimensional Banach manifold $\ca{M}$ equipped with:
\begin{itemize}
\item an $LG$-invariant closed $2$-form $\Omega$ such that the induced linear map $T_p\ca{M}\to T^*_p\ca{M}$ is injective.
\item a proper equivariant momentum map $\Phi:\ca{M}\to L\fra{g}^*$ satisfying that $d\innpro{\Phi}{\xi}{}=\iota_{\xi_\ca{M}}\Omega$ for $\xi\in L\fra{g}$, where $LG$ acts on $L\fra{g}^*$ via the gauge action, $\xi_\ca{M}$ is the vector field generated by $\xi$.
\end{itemize}
\end{dfn}
It has a compact model $\ca{M}/\Omega G$. Since $\Omega G$ acts on $L\fra{g}^*$ freely, the restriction of the $LG$-action on $\ca{M}$ to $\Omega G$ is also free. Since $\Phi$ is proper, the quotient space is compact. Unfortunately, it can be not a Hamiltonian $G$-space, but a quasi-Hamiltonian $G$-space defined below.
\begin{dfn}[\cite{AMM}]
A manifold $M$ is a quasi-Hamiltonian $G$-space if it has an invariant $2$-form $\omega$ and an equivariant map $\phi:M\to G$ such that:
\begin{itemize}
\item $d\omega=-\phi^*\chi$, where $\chi$ is the canonical $3$-form on $G$ defined by an invariant inner product on $\fra{g}$.
\item $\iota_{\xi_M}=\frac{1}{2}\phi^*(\theta+\overline{\theta},\xi)$ for any $\xi\in\fra{g}$, where $\theta$ ($\overline{\theta}$) is the left- (right-) invariant Maurer-Cartan form respectively.
\item $\ker \omega_x=\{\xi_M(x)\mid \xi\in\ker({\rm Ad}_{\phi(x)}+1)\}$ at each $x\in M$.
\end{itemize}
$\phi$ is called a group valued momentum map.
\end{dfn}
They are in 1:1 correspondence as follows.
\begin{pro}[\cite{AMM}]\label{Prop equivalence Sect.2}
If $\ca{M}$ is a Hamiltonian $LG$-space, $\ca{M}/\Omega G$ is a quasi-Hamiltonian $G$-space. Conversely, if $M$ is a quasi-Hamiltonians $G$-space, $\phi^*(L\fra{g}^*\to G)$ is a Hamiltonian $LG$-space. In short, the following diagram commutes.
$$\begin{CD}
\ca{M} @>\Phi>> L\fra{g}^* \\
@VHolVV @VholVV \\
M @>\phi>> G
\end{CD}$$
$Hol:\ca{M}\to M$ is the quotient map.
\end{pro}

If an $LG^\tau$-equivariant pre-quantum line bundle can be defined, $\ca{M}$ is said to be at level $\tau$.
We quote the result in \cite{Son} related to the above.

\begin{pro}
For a Hamiltonian $LG$-space $\ca{M}$ at level $\tau$, we can define a quantization $Q(\ca{M},\Omega)\in R^\tau(LG)$.
\end{pro}

\section{Settings and division of the problem}\label{Section simplification}
In this section, we explain our settings and simplify our problems. Since we have written what we do at least intuitively in Section \ref{Section Introduction}, we focus on more technical aspects.　In particular, we explain why we need a twisted crossed product.

The manifold which we study is as follows.
\begin{dfn}
Let $\ca{M}$ be an infinite dimensional Banach manifold equipped with an $LT$-action.
It is called a {\bf proper $LT$-space} if the restriction of the action to $\Omega T$ is free, the quotient space $M:=\ca{M}/\Omega T$ is compact and has a $T$-equivariant $Spin^c$-structure, and the principal bundle
$$\Omega T_0\to \ca{M}\to \ca{M}/\Omega T_0=:\widetilde{M}$$
is trivial. We always use this decomposition to study $\ca{M}$. The quotient space $M$ is called a {\bf compact model} of $\ca{M}$ and $\widetilde{M}$ is called a {\bf locally compact model} of $\ca{M}$.
\end{dfn}
One may feel the definition is artificial a little. In fact, we believe we can extend the results to more general cases. However, we do not make an effort to generalize in this paper. Let us show some examples which tell us that the above package is worth studying.

\begin{exs}
$(1)$ The set of connections on the trivial bundle $\bb{T}_\rot\times T$ can be identified with $L\fra{t}^*=\Omega^1(\bb{T}_\rot,\fra{t})$. It is equipped with the gauge action of $LT$, and is a proper $LT$-space as follows.
The compact model is $T$ and the quotient map to $T$ is the holonomy map $hol:L\fra{t}^*\to T$, where we regard $L\fra{t}^*$ as the set of connections on the trivial bundle $\bb{T}_\rot\times T$. Since the holonomy map factors through
$$L\fra{t}^*\xrightarrow{\int}\fra{t}\xrightarrow{\exp}T,$$
the locally compact model is $\fra{t}$.
The parallel connections $d+v$ for $v\in\fra{t}$ trivialize the principal bundle
$\Omega T_0\to L\fra{t}^*\to \fra{t}$.

$(2)$ The left translation on $LT$ is an example of the above definition. The compact model is $T$ and the locally compact model is $T\times \Pi_T$. The both quotient maps are the natural ones induced from the canonical decomposition $LT=T\times \Pi_T\times \Omega T_0$.

$(3)$ An $LT$-Hamiltonian space is also an example. In fact, our setting is an obvious generalization of this type.
The restriction of the $LT$-action on $\ca{M}$ to $\Omega T$ is free because of the equivariance of $\Phi$ which is the momentum map. The quotient $\ca{M}/\Omega T$ is compact since $\Phi$ is proper.
The trivialization of $\ca{M}\to\ca{M}/\Omega T_0$ can be defined by the pull back of the one of $L\fra{t}^*\to \fra{t}$ explained in $(1)$.
The group valued moment map $\phi:M\to T$ constructed in Proposition \ref{Prop equivalence Sect.2}
 is in fact a circle valued moment map introduced by \cite{McD} in this case.
\end{exs}

In order to define a Dirac operator, we need to deal with a Clifford module bundle. In order to define an analytic index as an element of $R^\tau(LT)$, we must deal with not an ordinary $LT$-equivariant Clifford module bundle, but a $\tau$-twisted one. In this paper, we deal with only the following type.

\begin{dfn}
$(1)$ Let $\ca{L}$ be a $\tau$-twisted equivariant line bundle, that is, $l\in LT$ determines a linear isomorphism $\sigma(l):\ca{L}|_m\to\ca{L}|_{l.m}$ satisfying that $\sigma(l_1)\sigma(l_2)=\tau(l_1,l_2)\sigma(l_1l_2)$. 

$(2)$ Let $\ca{S}:=S(T\widetilde{M})\boxtimes S(\Omega T_0)$ be the Spinor bundle. The first factor is well-defined since we assume that $M$ is $Spin^c$ and so is $\widetilde{M}$. The second one is the trivial bundle whose fiber is the Spinor defined in Def \ref{Dfn Spinor Sect.4},
 which comes from Section 3 in \cite{FHTII}. We regard that $\widetilde{M}$ and $\Omega T_0$ are orthogonal to one another, and also regard $\Omega T_0$ as an even dimensional manifold.

$(3)$ The tensor product $\ca{L}\otimes\ca{S}$ is our Clifford module bundle we deal with in this paper.
\end{dfn}

We finish the setting here. We will construct our objects ``$LT\ltimes_\tau C_0(\ca{M})$'' and ``$L^2(\ca{M},\ca{L}\otimes\ca{S})$'' using inductive limits. More precisely, we will explain the finite dimensional approximation of $\Omega T_0$:
$$U_1\to U_2\to\cdots\to U:=\Omega T_0,$$
and we use it for the construction as follows.

\begin{dfn}
$LT\ltimes_\tau C_0(\ca{M})$ is the inductive limit of the sequence
$$(T\times\Pi_T\times U_1)\ltimes_\tau C_0(\widetilde{M}\times U_1)\xrightarrow{i_1}(T\times\Pi_T\times U_2)\ltimes_\tau C_0(\widetilde{M}\times U_2)\xrightarrow{i_2}\cdots.$$
\end{dfn}
The construction of $i_1,i_2,\cdots$ is one of the difficulties of this paper. It will be done through the representation of each algebra.

\begin{dfn}
$L^2(\ca{M},\ca{L}\otimes\ca{S})$ is the inductive limit of the sequence
$$L^2(\widetilde{M}\times U_1,\ca{L}\otimes S(T\widetilde{M})\otimes S(U_1))\to L^2(\widetilde{M}\times U_2,\ca{L}\otimes S(T\widetilde{M})\otimes S(U_2))\to \cdots.$$
\end{dfn}
The construction of homomorphisms will be done using representation of $U_1,U_2,\cdots$.

Let us simplify the problems by using the property of our $2$-cocycle. For this purpose, we use the following. $\Gamma$ and $A$ are as around Definition \ref{Dfn norms Sect.2}.

\begin{lemma}
Let $\Gamma_1$ and $\Gamma_2$ be amenable groups, $A_1$ and $A_2$ be $C^*$-algebras, $\Gamma_i$ act on $A_i$ for $i=1,2$. Suppose that either $A_1$ or $A_2$ and either $\Gamma_1\ltimes_{\sigma_1} A_1$ or $\Gamma_2\ltimes_{\sigma_2} A_2$ are nuclear, that is, we do not face any troubles to define tensor products. Then,
$$\Gamma\ltimes A\cong(\Gamma_1\ltimes A_1)\otimes (\Gamma_2\ltimes A_2).$$
\end{lemma}
\begin{proof}
If we verify that
$$\Gamma\ltimes_{{\rm red}} A\cong(\Gamma_1\ltimes_{{\rm red}} A_1)\otimes_{{\rm min}} (\Gamma_2\ltimes_{{\rm red}} A_2),$$
amenability and nuclearity imply the result.

The above can be verified as follows. Let $\widetilde{\pi}\rtimes L$, $\widetilde{\pi}_1\rtimes L_1$ and $\widetilde{\pi}_2\rtimes L_2$ be the canonical covariant representation of $A$, $A_1$ and $A_2$ respectively explained in the above of Definition \ref{Dfn norms Sect.2}. We use a notation $\pi:=\pi_1\otimes\pi_2$ and $L:=L_1\otimes L_2$.
The algebraic tensor product
$$C_c(\Gamma_1,A_2)\otimes^\alg C_c(\Gamma_2,A_2)$$
is mapped to a dense subalgebra of $C_c(\Gamma,A)$ and the restriction of $\widetilde{\pi}\rtimes L$ to it is the tensor product of $\widetilde{\pi}_1\rtimes L_1$ and $\widetilde{\pi}_2\rtimes L_2$. The fact that the image is dense implies the completions coincide.
From the definition of the minimal tensor product, we obtain the result.
\end{proof}

We can obtain the twisted version of it thanks to the observation in the bottom of Section \ref{Subsection Cstar algebra} 
if we assume that $\sigma(\Gamma_1,\Gamma_2)=1$. Our cocycle $\tau$ satisfies it when $\Gamma_1=T\times\Pi_T$ and $\Gamma_2=U_N$, and hence we obtain the following.
\begin{pro}\label{Prop division Sect.3}
$$(T\times\Pi_T\times U_N)\ltimes_\tau C_0(\widetilde{M}\times U_N)$$
$$\cong (T\times\Pi_T)\ltimes_\tau C_0(\widetilde{M})\bigotimes U_N\ltimes_\tau C_0(U_N).$$
\end{pro}

About $L^2$-spaces, we can simplify much more easily.
The following is because the restriction of $\ca{L}$ to $\Omega T_0$ is trivial.
\begin{pro}
$$L^2(\widetilde{M}\times U_N,\ca{L}\otimes S(T\widetilde{M})\otimes S(U_N))$$
$$\cong L^2\bigl(\widetilde{M},\ca{L}\otimes S(T\widetilde{M})\bigr)\bigotimes \Bigl(L^2(U_N)\otimes S(U_N)\Bigr)$$
as representation spaces of $T\times \Pi_T\times U_N$.
\end{pro}

Therefore, it is enough to consider $\Omega T_0\ltimes_\tau \Omega T_0=\varinjlim U_N\ltimes_\tau C_0(U_N)$ and $L^2(\Omega T_0)=\varinjlim L^2(U_N)$ to construct $LT\ltimes_\tau C_0(\ca{M})$ and $L^2(\ca{M},\ca{L}\otimes\ca{S})$.

Let us move to the study of Dirac operator and its index.

We regard $\widetilde{M}$ and $\Omega T_0$ as orthogonal to one another. Therefore we define a Dirac operator $\ca{D}$ by
$$\ca{D}:=\cancel{D}\otimes\id +\id\otimes \cancel{\partial},$$
where $\cancel{D}$ is a $(T\times\Pi_T)^\tau$-equivariant Dirac operator on $\widetilde{M}$ and hence it makes sense. $\cancel{\partial}$ is what we will construct in the next section by using the algebraic Dirac operator constructed in \cite{FHTII}.

We will define the index by
$$\ind(\ca{D}):=\ind(\cancel{D})\otimes\ind(\cancel{\partial})$$
 as an element of $R^\tau(LT)$. We define the index of $\cancel{\partial}$ by an explicit computation as a formal difference of P.E.R. of $\Omega T_0$.
According to \cite{Kas}, $\cancel{D}$ has a well-defined index valued in $K_0(C^*((T\times\Pi_T)^\tau))$. 
To verify that the tensor product of two indices can be regarded as a formal difference of P.E.R. of $LT$, we need to compute $K_0(C^*((T\times\Pi_T)^\tau))$ carefully. It will be done in Section \ref{Section Index}.


\section{The $L^2$-space and the Dirac operator}\label{Section Hilbert}
In this section, we concentrate on the $\Omega T_0$-part, that is, we construct a Hilbert space $L^2(\Omega T_0)\otimes S$ and a Dirac operator acting on it. While the definitions are quite simple, it is not clear that the definition is reasonable. To show that the definition is not just an analogue of compact Lie groups but a limit of finite dimensional case, we study representations of $\Omega T_0$ and the finite dimensional approximation of it.
Then, we study $L^2$-spaces of finite dimensional approximation of $\Omega T_0$ by using the representations.

\subsection{P.E.R. of $\Omega T_0$}
Using $\Omega T_0\cong \{l\in C^\infty(\bb{T}_\rot,\fra{t})\mid\int ld\theta=0\}$ $f\mapsto f'$, we describe the representation as a projective representation of linear space. Let $U$ be the right hand side for simplicity.
$d$ is the generator of $\bb{T}_{{\rm rot}}$-action and define the complex structure on ${\rm Lie}(U)$ as
$J:=-d/|d|.$ Let $\innpro{\cdot}{\cdot}{\tau}:\fra{t}\times\fra{t}\to\bb{R}$ be the inner product constructed in the following of the Proposition \ref{Prop cocycle Sect.2}.
Since $U$ is simply connected, it is enough to study representations of the Lie algebra ${\rm Lie}(U)$, and hence we focus on it.

We introduce a symplectic form on ${{\rm Lie}}(U)$ by
$$\omega(f,g):=\int \innpro{f}{\frac{dg}{d\theta}}{\tau}d\theta$$
for $f,g\in {\rm Lie}(U)$. Let ${\bf 1}\in\fra{t}$ be a unit vector.
The symplectic base is the set
$$\frac{\cos\theta}{\sqrt{\pi}}{\bf 1},\frac{\sin\theta}{\sqrt{\pi}}{\bf 1},\cosi{2}{\bf 1},\sine{2}{\bf 1},\cdots$$
and $J$ acts on this basis as
$$J\bigl(\cosi{k}{\bf 1}\bigr)=\sine{k}{\bf 1},\; J\bigl(\sine{k}{\bf 1}\bigr)=-\cosi{k}{\bf 1}.$$
We choose the inner product $g$ defined by a usual formula $g(X,Y):=\omega(X,J(Y))$. Hence the symplectic basis is also an orthonomal basis.

\begin{rmks}
$(1)$ From now on, we omit the symbol ${\bf 1}$ for simplicity. It is possible because we deal with only a circle.

$(2)$ Usually, a terminology ``level'' appears in the representation theory of loop groups. However, in our setting, information of level is included in the inner product $\innpro{\cdot}{\cdot}{\tau}$ and does not appear explicitly.
\end{rmks}

We take the complex orthonormal basis on $\Lie(U)\otimes \bb{C}$ by
$$z_k:= \frac{1}{\sqrt{2}}\bigl(\cosi{k}+i\sine{k}\bigr),\;\overline{z_k}:= \frac{1}{\sqrt{2}}\bigl(\cosi{k}-i\sine{k}\bigr).$$
Let us notice that $d(z_k)=ikz_k$, $d(\overline{z_k})=-ik\overline{z_k}$. We call a linear combination of such vectors a {\it finite energy vector}.

Let us recall that a $\tau$-twisted representation $\rho$ is a continuous map from $\Omega T_0$ to $U(\ca{H})$ satisfying that
$\rho(v_1)\rho(v_2)=\rho(v_1+v_2)\tau(v_1,v_2).$
We are interested in the case when $\tau(v_1,v_2)=\exp(i\omega(v_1,v_2))$. Taking the infinitesimal representation, the relation implies the commutation relation
$$[d\rho(X),d\rho(Y)]=2i\omega(X,Y).$$
It tells us that the representation of $\Omega T_0$ heavily relates to the infinite particles CCR and hence quantum field theory.

This viewpoint suggests that we should define the representation by
$$d\rho(\cosi{k})=\frac{\partial}{\partial x_k},\; d\rho(\sine{k})=2ix_k$$
on ``$L^2(\bb{R}^\infty)$.'' Needless to say, this Hilbert space itself does NOT make sense. Rather, $L^2(\bb{R}^\infty)$ {\bf IS DEFINED} as the representation space.

Let us construct an irreducible P.E.R. $d\rho$ of $\Lie(U)$, that is, we define operators $d\rho(z_k)$, $d\rho(\overline{z_k})$ and $d\rho(d)$.

Let $\widehat{S}\bigl((U\otimes\bb{C})_+\bigr)$ be the completion of the symmetric algebra of 
$(U\otimes\bb{C})_+:=\bigoplus_{k\in \bb{N}}\bb{C}z_k$ with respect to the inner product defined by a little strange formula
$$(z_{i_1},z_{i_2},\cdots,z_{i_l},z_{j_1},z_{j_2},\cdots,z_{j_l})=\sum_{\sigma\in\fra{S}_l}\delta_{i_1,\sigma(j_1)}\cdots\delta_{i_l,\sigma(j_l)}$$
or more explicit formula
$$(z_1^{n_1}z_2^{n_2}\cdots,z_1^{m_1}z_2^{m_2}\cdots)=n_1!n_2!\cdots\delta_{n_1,m_1}\delta_{n_2,m_2}\cdots.$$
We sometimes write $1$ as $\Omega$ in the following, and call it the {\it lowest weight vector}.

First of all, $d\rho(d)$ is defined by the formula
$$d\rho(d)(z_1^{n_1}z_2^{n_2}\cdots):=i(\sum_jjn_j)z_1^{n_1}z_2^{n_2}\cdots.$$
A finite linear combination of such eigenvectors is also called a {\it finite energy vector}. 

Let us define the operators $d\rho(z_k)$ and $d\rho(\overline{z_k})$.
We set
$$d\rho(z_k):=-\sqrt{2}M_{z_k},\;d\rho(\overline{z_k})=\sqrt{2}\frac{\partial}{\partial z_k},$$
where $M_{z_k}$ is the multiplication operator with $z_k$. One can check that the operators satisfy the commutation relation by an easy computation.

\begin{lemma}
$d\rho$ is a unitary representation on $\Lie(U)$, that is, $d\rho(X)$ is skew adjoint for any $X\in{\rm Lie}(U)$.
\end{lemma}
\begin{proof}
Since $d\rho(\cosi{k})=d\rho(\frac{z_k+\overline{z_k}}{\sqrt{2}})$ and $d\rho(\sine{k})=d\rho(\frac{z_k-\overline{z_k}}{\sqrt{2}i}$), it is sufficient to verify that $M_{z_k}^*=\frac{\partial}{\partial z_k}$. It can be verified as follows.
$$(M_{z_k}z_1^{n_1}\cdots z_k^{n_k}\cdots,z_1^{n_1}\cdots z_k^{n_k+1}\cdots)=(n_k+1)\Pi_j n_j!$$
$$(z_1^{n_1}\cdots z_k^{n_k}\cdots,\frac{\partial}{\partial z_k}z_1^{n_1}\cdots z_k^{n_k+1}\cdots)=(n_k+1)\Pi_j n_j!$$
\end{proof}

The following is a convenient formula to deal with representations.

\begin{lemma}
\begin{eqnarray*}
d\rho(d)&=&-\frac{i}{2}\sum_k kd\rho(z_k)d\rho(\overline{z_k}) \\
&=& -\frac{i}{2}\sum_kd\rho(\sqrt{k}z_k)d\rho(\sqrt{k}\overline{z_k})
\end{eqnarray*}
\end{lemma}
\begin{rmk}
The coefficient $k$ in the first row is slightly artificial and the $\sqrt{k}$ in the second is also. It is because we choose not $\sqrt{k}z_k$'s but $z_k$'s  as the basis. The former is a orthogonal basis with respect to $\Innpro{\cdot}{\cdot}{\tau}$ and Freed, Hopkins and Teleman used this basis in order to define the Dirac operator in \cite{FHTII}.
\end{rmk}
This lemma allows us to check the commutation relation in the universal enveloping algebra of $\bigl({\rm Lie}(U)\oplus \bb{R}d\oplus\bb{R}K\bigr)\otimes\bb{C}$, where $K$ is an added element such that $[z_k,\overline{z_k}]=2K$, $[K,X]=0$ for any $X$, and $[d,Y]:=dY$ for any $Y\in \Lie(U)$. Since $K$ acts on the representation space by $i\id$, we omit $K/i$ in the following, for example, $[z_k,\overline{z_k}]=2$.
 It is sufficient to verify that
$$[-\frac{i}{2}\sum_kkz_k\overline{z_k},z_p]=dz_p=ipz_p$$
$$[-\frac{i}{2}\sum_kkz_k\overline{z_k},\overline{z_p}]=d\overline{z_p}=-ip\overline{z_p}.$$
Let us check only the first equality.
\begin{eqnarray*}
[-\frac{i}{2}\sum_jjz_k\overline{z_k},z_p] &=& -\frac{i}{2}[pz_p\overline{z_p},z_p] \\
&=&-\frac{i}{2}pz_p[\overline{z_p},z_p] \\
&=& -\frac{i}{2}pz_p\cdot (-2) \\
&=& ipz_p
\end{eqnarray*}
thanks to the Leibniz property of the Lie bracket.
Combining all of the above, we finish the construction of P.E.R. of $\Omega T_0$.
It is known that the representation constructed above is the unique irreducible P.E.R. of $\Omega T_0$ at level $\tau$.

We introduce a slightly strange notation.
\begin{dfn}
$$L^2(\bb{R}^\infty):=\widehat{S}\bigl((U\otimes\bb{C})_+\bigr),$$
where $\widehat{S}$ denotes the completion of the symmetric algebra.
\end{dfn}

\subsection{The construction of $L^2(\Omega T_0)$ and its justification}

WZW Hilbert space for a loop group $LG$ is defined by
$$\bigoplus_{\lambda\in \widehat{LG^\tau}}V_\lambda^*\otimes V_\lambda$$
as an analogue of Peter-Weyl's theorem, where $\widehat{LG^\tau}$ is the set of isomorphism classes of P.E.R.s of $LG$ at level $\tau$. 
It is studied in conformal field theory and the name WZW comes from three physicists Wess, Zumino and Witten. One can consult with \cite{Gaw} about the theory.
$\ca{H}_{{\rm WZW}}$ is regarded as the Hilbert space consisting of $L^2$-sections of the line bundle on $LG$ defined by the central extension $\tau$.
For $\Omega T_0$, P.E.R. is unique and we don't need the $\oplus$. We define a Hilbert space following the above.
\begin{dfn}
We set the WZW Hilbert space for $\Omega T_0$ as
$$\ca{H}_{{\rm WZW}}:=L^2(\bb{R}^\infty)^*\otimes L^2(\bb{R}^\infty).$$
$L^2(\bb{R}^\infty)^*$ is the adjoint representation space of $L^2(\bb{R}^\infty)$; the action $\rho^*$ is defined by the formula
$\innpro{\rho^*_g(\phi)}{\psi}{}:=\innpro{\phi}{\rho_{-g}(\psi)}{}.$ It is a negative energy representation whose representation can be guessed easily. For example, $\frac{1}{i}d\rho(d)$ is negative definite and $L^2(\bb{R}^\infty)^*$ is generated by the highest weight vector.

We sometimes write $\ca{H}_{{\rm WZW}}$ as $L^2(\Omega T_0)$.
\end{dfn}
An element $\phi\otimes \psi\in\ca{H}_{{\rm WZW}}$ determines a ``function'' on $\Omega T_0$ by the formula
$$f\otimes v(g):=\Bigl(\frac{\sqrt{2}}{\pi}\Bigr)^\infty\innpro{\phi}{\rho_{-g}(\psi)}{},$$
where $\innpro{\cdot}{\cdot}{}$ is the natural pairing.
The crazy normalization factor $(\sqrt{2}/\pi)^\infty$ is explained below.
We cannot define a {\bf VALUE} of this ``function'' at any point. It means that the notation $L^2(\Omega T_0)$ has not been justified yet.
For a compact group, such a function is square integrable. In our case, $\Omega T_0$ does not have a nice measure and hence it is nonsense to wonder whether it it also $L^2$.

In this subsection, we ``justify'' this model by using inductive limits. More precisely, we study $L^2(U_N)$'s in terms of representation theory, and study the inductive limit of these spaces, where $\{U_N\}$ is the finite dimensional approximation described below.

\begin{dfn}
$$U_{N}:={\rm Span}_\bb{R}\Bigl\{\frac{\cos\theta}{\sqrt{\pi}},\frac{\sin\theta}{\sqrt{\pi}},\cosi{2},\sine{2},\cdots,\cosi{N},\sine{N}\Bigr\}$$
Each of the restrictions of $J$, $d$, $g$, $\omega$ and $\tau$ makes suitable sense (for example, $\omega$ is nondegenerate) and we use the same symbols for the restrictions of them. We define $\widehat{S}\bigl((U_{N}\otimes\bb{C})_+\bigr)$, $\rho_{N}:U_{N}\to U\Bigl(\widehat{S}\bigl((U_{N}\otimes\bb{C})_+)\bigr)\Bigr)$, $d\rho_N(d)$, $z_k$ and $\overline{z_k}$ in the same manner.
\end{dfn}
The following is obvious from the definition and the above construction
\begin{pro}
The natural embedding $(x_1,y_1,\cdots,x_p,y_p)\mapsto(x_1,y_1,\cdots,x_p,y_p,0,0)$ defines a sequence
$$U_1\xrightarrow{i_1} U_2\xrightarrow{i_2}\cdots \xrightarrow{i_{N-1}} U_{N}\xrightarrow{i_N}\cdots,$$
and it induces a sequence
$$ \widehat{S}\bigl((U_{1}\otimes\bb{C})_+\bigr)\xrightarrow{\Phi_1}\widehat{S}\bigl((U_{2}\otimes\bb{C})_+\bigr)\xrightarrow{\Phi_2}\cdots\xrightarrow{\Phi_{N-1}}\widehat{S}\bigl((U_{N}\otimes\bb{C})_+\bigr)\xrightarrow{\Phi_N}\cdots.$$
They satisfy that
$$\Phi_N\circ\rho_{N}(g)=\rho_{N+1}(i_N(g))\circ\Phi_N.$$
\end{pro}
In general, inductive limit of Hilbert space in the algebraic sense is not always a Hilbert space. For example, the inductive limit of $\bb{R}\to\bb{R}^2\to\cdots$ is not complete. However, we always take the completion of the inductive limit of Hilbert space like the case of $C^*$-algebras. The inductive sequence $U_1\xrightarrow{i_1} U_2\xrightarrow{i_2}\cdots$ is like this and we take a completion in the $C^\infty$-topology in $U$.

Using the explicit description of the representation, we obtain the following. 
\begin{thm}
$$\rho:U\to U\Bigl(\widehat{S}\bigl(U\otimes\bb{C})_+\bigr)\Bigr)$$
$$=\varinjlim\biggl(\rho_N:U_N\to U\Bigl(\widehat{S}\bigl((U_N\otimes\bb{C})_+\bigr)\Bigr)\biggr)$$

\end{thm}
We need to explain the definition of $\varinjlim \rho_N$. Let $j_N:U_N\to U$ and $\Psi_N:\widehat{S}\bigl((U_N\otimes\bb{C})_+\bigr)\to \widehat{S}\bigl((U\otimes\bb{C})_+\bigr)$ be the natural homomorphisms defined by the inductive limit. 
Due to the continuity, it is sufficient to deal with $g\in U_N$ and $v\in \widehat{S}\bigl((U_M\otimes\bb{C})_+\bigr)$. Moreover, we may assume that $M=N$ by using $i_n$'s or $\Phi_n$'s.
$$\bigl(\varinjlim\rho_N\bigr)_{j_N(g)}\Bigl(\Psi_N(v)\Bigr):=\Psi_N\Bigl(\rho_N(g)(v)\Bigr)$$
is the definition of the inductive limit of representations.

The Hilbert space $\widehat{S}\bigl((U_N\otimes\bb{C})_+\bigr)$ inside the $\varinjlim$ is isomorphic to $L^2(\bb{R}^N)$ as in \cite{Kir}. This is why, we introduce the notation $L^2(\bb{R}^\infty)$.
The only nontrivial point is the definition of the counterpart of the homomorphism $\widehat{S}\bigl((U_N\otimes\bb{C})_+\bigr)\to\widehat{S}\bigl((U_{N+1}\otimes\bb{C})_+\bigr)$. It is defined by a map
$$f\mapsto f\otimes \Bigl(\frac{2}{\pi}\Bigr)^{\frac{1}{4}}e^{-x^2}.$$
$e^{-x^2}$ is the lowest weight vector of the representation of $U_1$.

Let us study $L^2(U_N)$ in terms of matrix coefficients. It has two twisted actions $L$ and $R$ of $U_N$. $L$ is $\tau$-twisted and $R$ is $\tau^{-1}$-twisted defined by the formula
$$L_g(f)(x):=f(x-g)\tau(g,x),\; R_g(f)(x):=f(x+g)\tau(g,x).$$

We set a matrix coefficient by the formula
$$\phi\otimes\psi(x):=\Bigl(\frac{\sqrt{2}}{\pi}\Bigr)^N\innpro{\phi}{\rho_{-x}(\psi)}{}.$$
The crazy factor above comes from this formula.
We have not verified the RHS is truly square integrable. However, it is verified as a corollary of the following. Let $\Psi$ be the map defined by the above matrix coefficients.

\begin{thm}
$(1)$ $\Psi\circ(\id\otimes\rho_g)=L_g\circ\Psi$ and $\Psi\circ(\rho^*_g\otimes\id)=R_g\circ\Psi$.

$(2)$ If $\phi$ and $\psi$ are finite energy vectors, $\Psi(\phi\otimes\psi)$ can be written as a linear combination of Hermite polynomials.

$(3)$ $\Psi:L^2(\bb{R}^N)^*\otimes L^2(\bb{R}^N)\to L^2(U_N)$ is an isometric isomorphism as representation spaces.
\end{thm}
\begin{proof}
$(1)$ can be verified as follows. We verify here $R$-side only. 
\begin{eqnarray*}
\Phi(\rho^*_g(\phi)\otimes\psi)(x) &=& \Bigl(\frac{\sqrt{2}}{\pi}\Bigr)^N\innpro{\rho^*_g(\phi)}{\rho_{-x}(\psi)}{} \\
&=&\Bigl(\frac{\sqrt{2}}{\pi}\Bigr)^N\innpro{\phi}{\rho_{-g}\bigl(\rho_{-x}(\psi)\bigr)}{} \\
&=&\Bigl(\frac{\sqrt{2}}{\pi}\Bigr)^N\innpro{\phi}{\rho_{-g-x}(\psi)\tau(-g,-x)}{} \\
&=&\Psi(\phi\otimes\psi)(x+g)\tau(g,x) \\
&=&R_g\circ\Psi(\phi\otimes\psi)(x)
\end{eqnarray*}

In order to verify $(2)$ and $(3)$, it is enough to study when $N=1$. Hence we can omit the subscript in $z_k$'s. The result can be verified by direct computation using the following.
\begin{lemma}
Let $F\in C_c^\infty(U_1)$.
\begin{eqnarray*}
dR(e_1)(F)(x) &=& \frac{\partial F}{\partial x_1}(x)+ix_2F(x),\\
dR(e_2)(F)(x) &=& \frac{\partial F}{\partial x_2}(x)-ix_1F(x), \\
dL(e_1)(F)(x) &=& -\frac{\partial F}{\partial x_1}(x)+ix_2F(x),\\
dL(e_2)(F)(x) &=& -\frac{\partial F}{\partial x_2}(x)-ix_1F(x).
\end{eqnarray*}
\end{lemma}

Let $a_j:=\partial/\partial x_j+x_j$ be the annihilation operator, and the creation operator $a_j^*$ be its dual.
We obtain the following from the above:
\begin{itemize}
\item $dR(z)=\frac{1}{\sqrt{2}}(a_1+ia_2)$ and $dR(\overline{z})=\frac{1}{\sqrt{2}}(-a_1^*+ia_2^*)$.
\item $dL(z)=\frac{1}{\sqrt{2}}(a_1^*+ia_2^*)$ and $dL(\overline{z})=\frac{1}{\sqrt{2}}(-a_1+ia_2).$
\item $[dL(z),dL(\overline{z})]=[dR(\overline{z}),dR(z)]=2\id$.
\item $dR$'s commute with $dL$'s. 
\end{itemize}
In particular, we can describe creation/annihilation operators in terms of representations.
Therefore, if one notice that $\Omega:=\Psi(1\otimes1)$ is a Gaussian function, $\Psi(d\rho^*(\overline{z})^k1\otimes d\rho(z)^l1)$ can be written as a linear combination of Hermite polynomials.


To verify $(3)$, we use Hermite polynomial technique which we learned in \cite{Ara}. 
We can compute the inner product between $dR(\overline{z})^kdL(z)^l\Omega$ with $dR(\overline{z})^mdL(z)^n\Omega$. Before that, we notice the following computation
\begin{eqnarray*}
dR(z)^nR(\overline{z})^k &=& dR(z)^{n-1}\Bigl([dR(z),dR(\overline{z})]+dR(\overline{z})dR(z)\Bigr)dR(\overline{z})^{k-1} \\
&=&-2dR(z)^{n-1}dR(\overline{z})^{k-1} + dR(z)^{n-1}dR(\overline{z})dR(z)dR(\overline{z})^{k-1}\\
&=& \cdots\\
&=& -2kdR(z)^{n-1}dR(\overline{z})^{k-1} + {\rm something}\cdot dR(z) \\
&=& \cdots \\
&=& \begin{cases} \text{something}\cdot dR(z) & k<n \\
(-2)^kk! + \text{something}\cdot dR(z) & k=n \\
dR(\overline{z})\cdot \text{something} & k>n
\end{cases}
\end{eqnarray*}
and a similar formula for $dL$'s. The inner product can be computed as follows. Notice that $dR(z)$ and $dL(\overline{z})$ kill $\Omega$.
\begin{eqnarray*}
&&\innpro{dR(\overline{z})^kdL(z)^l\Omega}{dR(\overline{z})^mdL(z)^n\Omega}{} \\
\;\;&=& (-1)^{k+l}\innpro{\Omega}{dL(\overline{z})^ldR(z)^kR(\overline{z})^mdL(z)^n\Omega}{} \\
\;\;&=&\begin{cases} (-1)^{k+l}\innpro{\Omega}{\cdots dR(z)\Omega}{} & k<n \\
(-1)^l2^kk!\innpro{\Omega}{dL(\overline{z})^ldL(z)^n\Omega}{} & k=n \\
(-1)^{k+l}\innpro{\Omega}{dR(\overline{z})\cdots \Omega}{} & k>n
\end{cases}\\
\;\;&=&\cdots \\
\;\;&=&2^kk!\cdot 2^ll!\delta_{k,m}\delta_{l,n}\innpro{\Omega}{\Omega}{}.
\end{eqnarray*}
It coincides with $\innpro{d\rho(\overline{z})^k1\otimes d\rho(z)^l1}{d\rho(\overline{z})^m1\otimes d\rho(z)^n1}{}$, hence $\Psi$ is isometric. Since the image is dense, $\Psi$ is bijective and we obtain the result.




\end{proof}

Therefore we can identify $L^2(U_N)$ with $L^2(\bb{R}^N)^*\otimes L^2(\bb{R}^N)$, and hence we can define an inductive system $\cdots\to L^2(U_N)\to L^2(U_{N+1})\to\cdots$ by tensoring with the lowest weight vector. The limit of this system is $\ca{H}_\WZW$ and we regard it as a justification of the definition $L^2(\Omega T_0):=\ca{H}_\WZW$.


\subsection{The algebraic Dirac operator}
We construct a Dirac operator $\cancel{\partial}$ on $\Omega T_0$ following \cite{FHTII}, and observe that the operator is reasonable if one accepts the definition $L^2(\Omega T_0)=L^2(\bb{R}^\infty)^*\otimes L^2(\bb{R}^\infty)$.

Let us recall the construction of the algebraic Dirac operator. The Spinor is the exterior algebra of an infinite dimensional complex vector space.
\begin{dfn}\label{Dfn Spinor Sect.4}
We set the dense subspace of the Spinor $S$ as
$$S_{{\rm fin}}:=\bigwedge^*_{{\rm alg}}\bigoplus^{{\rm alg}}_{k>0}\bb{C}\overline{z_k}.$$
The Clifford multiplication is defined by
\begin{eqnarray*}
\gamma(\overline{z_k})&:=&\sqrt{2}\overline{z_k}\wedge \\
\gamma(z_k)&:=&-\sqrt{2}\overline{z_k}\rfloor.
\end{eqnarray*}
$S_\fin$ has a standard inner product and a standard grading. $S$ is the completion of $S_\fin$ with respect to this inner product and inherits the grading as $S=S^0\widehat{\oplus} S^1$.
\end{dfn}
We define the number operator $N$ by the formula $N(\overline{z_{j_1}}\wedge \overline{z_{j_1}}\wedge \cdots \wedge\overline{z_{j_p}}):=(\sum j_k)\overline{z_{j_1}}\wedge \overline{z_{j_1}}\wedge \cdots \wedge\overline{z_{j_p}}$. Let us notice that $N=-\frac{1}{2}\sum k\gamma(z_k)\gamma(\overline{z_k})$.

The following operator acts on $L^2(\bb{R}^\infty)_\fin^*\otimes^\alg S_\fin$, which is essentially the Dirac operator in \cite{FHTII}.
\begin{dfn}
$$\cancel{d}:=i\sum \sqrt{k}d\rho^*(z_k)\otimes\gamma(\overline{z_k})+i\sum \sqrt{k}d\rho^*(\overline{z_k})\otimes\gamma(z_k).$$
\end{dfn}
Let us notice that an artificial coefficient $\sqrt{k}$ is due to the choice of the basis just like the definition of $d\rho^*(d)$. Since our Spinor is associated to our inner product $\omega(\cdot,J\cdot)$, such a coefficient does not appear in Clifford multiplication.

While Dirac operators act on the space of $L^2$-sections of Spinor bundles generally, $\cancel{d}$ acts on the ``square root'' of the $L^2$-space. We justify it by using the WZW model. We introduce another operator on $\ca{H}_{{\rm WZW}}\otimes S$, which looks like a Dirac operator more naturally after some computation.
\begin{dfn}
$$\cancel{\partial}:=i\sum \sqrt{k}d\rho^*(z_k)\otimes\id\otimes\gamma(\overline{z_k})+i\sum \sqrt{k}d\rho^*(\overline{z_k})\otimes\id\otimes\gamma(z_k).$$
If we rewrite $\ca{H}_\WZW$ as $L^2(\bb{R}^\infty)\otimes L^2(\bb{R}^\infty)^*$, $\cancel{\partial}$ can be written simply as $\id\otimes\cancel{d}$.
\end{dfn}
We will regard it as an $\Omega T_0$-equivariant Dirac operator on $\Omega T_0$, and compute the index valued in $R^\tau(\Omega T_0)$. Before that, we observe a corresponding object in the sub group $U_1$.

\begin{thm}
$(1)$ If we use the coordinate of $U_1$ and make $\cancel{\partial}$ act on $L^2(U_1)\otimes S(U_1)$,
\begin{eqnarray*}
\cancel{\partial}_{U_1} &=&i
\begin{pmatrix}
0 & -a_{x_1}^*+ia_{x_2}^* \\
a_{x_1}+ia_{x_2} & 0\end{pmatrix},
\end{eqnarray*}
where $a_{x_j}:=\frac{\partial}{\partial x_j}+x_j$, the annihilation operator. 

$(2)$ It is an elliptic operator.

$(3)$ It is $U_1$-equivariant with respect to the action $L$.
\end{thm}
\begin{proof}
For $(1)$, it is just a computation. We omit the subscript in $z_1$ in this proof. 
We use the notation $\partial_j:=\partial/\partial x_j$. Notice the formula
$$\gamma(z)=\begin{pmatrix} 0 & -\sqrt{2} \\ 0 & 0 \end{pmatrix}, 
\gamma(\overline{z})=\begin{pmatrix} 0 & 0 \\ \sqrt{2} & 0 \end{pmatrix}$$

\begin{eqnarray*}
dR(z)&=&\frac{1}{\sqrt{2}}\Bigl(dR(e_1)+idR(e_2)\Bigr)=\frac{1}{\sqrt{2}}\Bigl(\partial_1+ix_2+i\partial_2+x_1\Bigr) \\
&=&\frac{1}{\sqrt{2}}\Bigl(a_{x_1}+ia_{x_2}\Bigr) \\
dR(\overline{z})&=&\frac{1}{\sqrt{2}}\Bigl(dR(e_1)-idR(e_2)\Bigr)=\frac{1}{\sqrt{2}}\Bigl(\partial_1+ix_2-i\partial_2-x_1\Bigr) \\
&=&\frac{1}{\sqrt{2}}\Bigl(-a_{x_1}^*+ia_{x_2}^*\Bigr).
\end{eqnarray*}
Therefore
$$\cancel{\partial}_{U_1}=i\begin{pmatrix}
 0 & a_{x_1}^*-ia_{x_2}^* \\ a_{x_1}+ia_{x_2} & 0 \end{pmatrix}$$

$(2)$ is clear from the description of $(1)$. 

For $(3)$, what we need is that $\bigl(L(g)\otimes{\rm id}\bigr)\circ \cancel{\partial}_{U_1} \circ \bigl(L(g^{-1})\otimes{\rm id}\bigr)=\cancel{\partial}_{U_1}$. It is equivalent to that 
$$[dL(v)\otimes{\rm id},\cancel{\partial}_{U_1}]=0$$
for all $v\in {\rm Lie}(U_1)\otimes\bb{C}$ due to the connectedness of $U_1$. It suffices to check it when $v=z$ or $\overline{z}$.
\begin{eqnarray*}
[dL(z)\otimes\id,\cancel{\partial}_{U_1}] &=& \frac{i}{\sqrt{2}}\biggl[
\begin{pmatrix} 
a_{x_1}^*+ia_{x_2}^* & 0 \\ 
0 &a_{x_1}^*+ia_{x_2}^* 
\end{pmatrix},
\begin{pmatrix}
0 & -a_{x_1}^*+ia_{x_2}^* \\
a_{x_1}+ia_{x_2} & 0\end{pmatrix}\biggr]\\
&=& \frac{i}{\sqrt{2}}\begin{pmatrix} 0 & [a_{x_1}^*+ia_{x_2}^*,-a_{x_1}^*+ia_{x_2}^*] \\
[a_{x_1}^*+ia_{x_2}^*,a_{x_1}+ia_{x_2}] & 0
\end{pmatrix}\\
&=& 0
\end{eqnarray*}
The last equality is due to the commutation relation $[a_{x_i},a_{x_j}]=[a_{x_i}^*,a_{x_j}^*]=[a_{x_1},a_{x_2}^*]=[a_{x_1}^*,a_{x_2}]=0$ and $[a_{x_1},a_{x_1}^*]=[a_{x_2},a_{x_2}^*]=2$ for any $i,j$.

Since $(\cancel{\partial}_{U_1})^*=\cancel{\partial}_{U_1}$ and $dL(\overline{z_1})^*=-dL(z_1)$, we obtain that $[dL(\overline{z_1}),\cancel{\partial}_{U_1}]=0$.
\end{proof}
Since $(3)$ is valid for $\cancel{\partial}$, we can regard  it as an $\Omega T_0$-equivariant Dirac operator on $\Omega T_0$. It tells us that ${\rm Lie}(\Omega T_0)$, and hence $\Omega T_0$ acts on the kernel and the cokernel of $\cancel{\partial}$. Therefore, we can define an index 
$${\rm ind}(\cancel{\partial}):=[\ker(\cancel{\partial}|_{\ca{H}_{{\rm WZW}}\otimes S^0})]-[\ker(\cancel{\partial}|_{\ca{H}_{{\rm WZW}}\otimes S^1})]$$
as the formal difference of representations of $\Omega T_0$. Unlike the classical case when a compact group acts on a compact manifold, this index is doubtful whether it is finitely reducible.
Let us study it by an explicit computation of the index. Since $\cancel{\partial}={\rm id}\otimes \cancel{d}$,
$$\ker(\cancel{\partial})=\ker(\cancel{d})\otimes L^2(\bb{R}^\infty).$$

\begin{thm}
$$\ker(\cancel{d})=\bb{C}\Omega\widehat{\oplus} 0.$$
\end{thm}
\begin{proof}
We list formulas we need to relate $\cancel{d}^2$ with $d\rho^*(d)$.
\begin{itemize}
\item $d\rho^*(z_k)$'s and $d\rho^*(\overline{z_k})$'s commute with $\gamma(z_l)$'s and $\gamma(\overline{z_l})$'s.

\item $d\rho^*(z_k)$'s, $d\rho^*(\overline{z_k})$'s, $\gamma(z_l)$'s and $\gamma(\overline{z_l})$'s commute with themselves respectively.

\item $[d\rho^*(z_k),d\rho^*(\overline{z_l})]=[\gamma(z_k),\gamma(\overline{z_l})]=0$ if $k\neq l$.

\item $[d\rho^*(z_k),d\rho^*(\overline{z_k})]=-2{\rm id}_{L^2(\bb{R}^\infty)}$ and $[\gamma(z_k),\gamma(\overline{z_k})]=-2{\rm id}_S$.
\item $\gamma(z_k)^2=\gamma(\overline{z_k})^2=0$.
\end{itemize}
The readers should be careful that our commutater is graded, for example, $[\gamma(z_k),\gamma(\overline{z_k})]=\gamma(z_k)\gamma(\overline{z_k})+\gamma(\overline{z_k})\gamma(z_k)$.

\begin{eqnarray*}
\cancel{d}^2 &=& -\biggl(\sum \sqrt{k}d\rho^*(z_k)\otimes\gamma(\overline{z_k})\biggr)\biggl(\sum \sqrt{l}d\rho^*(z_l)\otimes\gamma(\overline{z_l})\biggr) \\
&&\;\;\;\;\;
-\biggl(\sum \sqrt{k}d\rho^*(z_k)\otimes\gamma(\overline{z_k})\biggr)\biggl(\sum \sqrt{l}d\rho^*(\overline{z_l})\otimes\gamma(z_l)\biggr) \\ 
&&  \;\;\;\;\;\;\;\;\;\;   -\biggl(\sum \sqrt{l}d\rho^*(\overline{z_l})\otimes\gamma(z_l)\biggr)\biggl(\sum \sqrt{k}d\rho^*(z_k)\otimes\gamma(\overline{z_k})\biggr)\\
&&\;\;\;\;\;\;\;\;\;\;\;\;\;\;\;-\biggl(\sum \sqrt{k}d\rho^*(\overline{z_k})\otimes\gamma(z_k)\biggr)\biggl(\sum \sqrt{l}d\rho^*(\overline{z_l})\otimes\gamma(z_l)\biggr)\\
&=& -\sum_{k\neq l}\sqrt{kl}d\rho^*(z_k)d\rho^*(\overline{z_l})\bigl[\gamma(\overline{z_k})\gamma(z_l)+\gamma(z_l)\gamma(\overline{z_k})\bigl] \\ 
&& \;\;\;\;\; -\sum_{k>0}k\bigl(d\rho^*(z_k) d\rho^*(\overline{z_k})\otimes\gamma(\overline{z_k})\gamma(z_k)+ d\rho^*(\overline{z_k})d\rho^*(z_k)\otimes\gamma(z_k)\gamma(\overline{z_k})
\bigr)\\
&=& -\sum_{k>0}k\bigl(d\rho^*(z_k) d\rho^*(\overline{z_k})\otimes\gamma(\overline{z_k})\gamma(z_k)+ d\rho^*(\overline{z_k})d\rho^*(z_k)\otimes\gamma(z_k)\gamma(\overline{z_k})
\bigr)\\
&=& -\sum_{k>0}k\cdot 2\gamma(\overline{z_k})\gamma(z_k)+\sum_{k>0}kd\rho^*(\overline{z_k}) d\rho^*(z_k)\otimes\bigl[\gamma(z_k)\gamma(\overline{z_k})+\gamma(\overline{z_k})\gamma(z_k)\bigr]\\
&=& 4\id\otimes N
-\frac{4}{i}d\rho^*(d).
\end{eqnarray*}
The both terms are positive definite, hence 
$$\ker(\cancel{d}^2)=\ker(d\rho^*(d))\otimes \ker(N).$$
$\ker(d\rho^*(d))$ is the $1$-dimensinal vector space spanned by the lowest weight vector $\Omega$. $\ker(N)$ is the $1$-dimensional space spanned by $1\in \bigwedge \bigoplus\bb{C}\overline{z_k}$. Combining them, we obtain the result.
\end{proof}

\section{The analytic index and the twisted group $C^*$-algebra}\label{Section Index}

Before defining the total index of $\ca{D}=\cancel{D}\otimes\id+\id\otimes\cancel{\partial}$ as an element of $R^\tau(LT)$, we need to introduce a twisted group $C^*$-algebras. According to Proposition \ref{Prop Kas Sect.1},
$\ind(\cancel{D})$ is an element of a $K$-group and our index $\ind(\cancel{\partial})$ is a representation. We need to rewrite $\ind(\cancel{D})$ as a formal difference of representations to be consistent. More precisely we verify that the $K$-group which has $\ind(\cancel{D})$ as an element is isomorphic to $R^\tau(T\times\Pi_T)$ defined below. It means that the index $\ind(\cancel{D})$ determines a formal difference of $\tau$-twisted representations of $T\times \Pi_T$. Therefore, our total index $\ind(\ca{D})$ can be defined as a formal difference of P.E.R. of $LT$ thus
 an element of $R^\tau(LT)$.

In addition, we define a twisted group $C^*$-algebra $C^*(LT,\tau)$ and verify that $K_0(C^*(LT,\tau))\cong R^\tau(LT)$. It is one of the circumstantial evidences so that our construction is reasonable.

\subsection{The analytic index valued in $K_0(C^*(T\times\Pi_T)^\tau)$}
We study the index of locally compact part in this and the next subsection by using the following result due to Kasparov
\begin{thm}[\cite{Kas88}, \cite{Kas}]\label{Thm Kas Sect.5}
Let $\Gamma$ be a locally compact group acting on a manifold $X$ properly and cocompactly, that is, the map $(x,\gamma)\mapsto (x,\gamma.x)$ is proper and the quotient space $X/\Gamma$ is compact. Let $D=D^+\oplus D^-$ be an $\Gamma$-equivariant Dirac operator on $X$.

Then, $D$ is $C^*(\Gamma)$-Fredholm and it has a well-defined index valued in $K_0(C^*(\Gamma))$.

\end{thm}
Even if we can verify that our setting satisfies the assumption above, we get an analytic index valued not in $K_0(C^*(T\times\Pi_T,\tau))$ but in $K_0(C^*((T\times\Pi_T)^*))$. We must improve it in the next subsection.

Let us check the assumption above here.
Through the homomorphism $(T\times\Pi_T)^\tau\to T\times \Pi_T$, $(T\times\Pi_T)^\tau$ acts on $\widetilde{M}$.
\begin{pro}
The action of $(T\times\Pi_T)^\tau$ on $\widetilde{M}$ is proper and cocompact.
\end{pro}
\begin{proof}
Let us recall the setting. $T$ acts on the compact model $M$ and $\widetilde{M}$ is a $\Pi_T$-bundle on $M$, that is, $\Pi_T$ acts on $\widetilde{M}$ freely. Therefore, we can take a fundamental domain $M'$ of the action of $\Pi_T$.

Let us take a compact set of the form $K\times K'\subseteq \widetilde{M}\times\widetilde{M}$. It is enough to deal with only the case when $K$ and $K'$ are subsets of $M'$. We need to verify that the set
$$Y_{K,K'}:=\{(x,\gamma)\in X\times (T\times\Pi_T)^\tau|x\in K, \gamma.x\in K'\}$$
is compact. Since $T$ is compact, $T.K$ is also. Therefore, there exists a finite subset $F\subseteq \Pi_T$ such that $T.K\subseteq \bigcup_{n\in F}n.K$. It tells us that
$$Y_{K,K'}\subseteq K\times (T\times F)^\tau.$$
Since the RHS is compact and the LHS is closed, $Y_{K,K'}$ is a compact set.
\end{proof}

\subsection{$C^*(T\times\Pi_T,\tau)$ and the analytic index}
We study the structure of the algebra $C^*((T\times\Pi_T)^\tau)$ and verify that our index $\ind(\cancel{D})$ is an element of $K_0(C^*(T\times\Pi_T,\tau))$, which is a desired result.
\begin{pro}
$$C^*((T\times\Pi_T)^\tau)\cong \bigoplus_{n\in \bb{Z}}C^*(T\times\Pi_T,n\tau),$$
where 
$$C^*(T\times\Pi_T,n\tau)\cong \bigoplus_{[\lambda]\in\Lambda_T/\kappa^{n\tau}(\Pi_T)}\ca{K}(V_{[\lambda]})$$
for any $n\neq0$, and
$$C^*(T\times\Pi_T,0)\cong C_0(\Pi_T\times T).$$
It is so called Pontryagin duality. Each of these algebras is defined in the proof.
\end{pro}
\begin{rmk}
We notice that if $\sigma$ is nonzero, it is injective. This is due to the assumption that $T$ is a circle, and hence for general torus, we must modify the statement. However, in the case when P.E.R.s are nice, $\sigma$ is injective and $C^*(T\times\Pi_T,\sigma)$ has a similar description.
\end{rmk}
\begin{proof}
Since the proof is slightly long, we summarize the proof here. Firstly we verify the decomposition in the first line of the statement computing the convolution product explicitly. Secondly, we determine all representations of $T\times \Pi_T$ at each level. Lastly, we study each direct summands which we study in the first step using the representation theory developed in the second step. More precisely, we study the induced representations of $C_c((T\times\Pi_T)^\tau)$ from representations of $(T\times\Pi_T)^\tau$.
We omit the proof of the last statement related to Pontryagin dual, because we do not need it in this paper.

Let us carry out the first step. For $f\in C_c^\infty((T\times\Pi_T)^\tau)$, we can decompose it as $f=\sum f_k(t,m)z^k$ by using Fourier series with respect to the direction of the added center. We use the coordinate $t\in T$, $m\in\Pi_T$ and $z\in U(1)$ for $(T\times\Pi_T)^\tau=T\times \Pi_T\times U(1)$ which is an isomorphism not as groups but as spaces. Let us compute the convolution product $f^1*f^2$. We use the formula
$(t,m,z)(t',m',z')=(tt',mm',zz'\kappa^\tau(m)(t'))$.
\begin{eqnarray*}
f^1*f^2(t,m,z) &=& \int f^1(t',m',z')f^2(t'^{-1}t,m'^{-1}m,z'^{-1}z\kappa^\tau(m')(t't^{-1}))dt'dm'dz' \\
&=&\int \sum f^1_k(t',m')z'^{k}f^2_l(t'^{-1}t,m'^{-1}m)(z'^{-1}z\kappa^\tau(m')(t't^{-1})^{-1})^ldt'dm'dz' \\
&=& \int \sum f^1_k(t',m')f^2_l(t'^{-1}t,m'^{-1}m)z^lz'^{k-l}\bigl(\kappa^\tau(m')(t't^{-1})\bigr)^ldt'dm'dz' \\
&=& \int \sum   f^1_k(t',m')f^2_k(t'^{-1}t,m'^{-1}m)z^k\bigl(\kappa^\tau(m')(t't^{-1})\bigr)^kdt'dm' \\
&=:& \sum f^1_k*_{-k}f^2(t,m)z^k.
\end{eqnarray*}
It tells us that the weight decomposition of $C_c^\infty((T\times\Pi_T)^\tau)=\bigoplus C_c^\infty(T\times\Pi_T)\otimes z^{-k}$ is in fact a decomposition as algebras with the above twisted convolutions in $C_c^\infty (T\times\Pi_T)\otimes z^{-k}$. $C_c^\infty(T\times\Pi_T,k\tau)$ denotes the algebra $C_c^\infty (T\times\Pi_T)\otimes z^{-k}$ equipped the twisted convolution $*_k$. We choose $-k$ as the subscript for convenience.
We finish the first step of the statement.

Let us move to the second step. 
Recall that $\tau((t,m),(t',m'))=\kappa^\tau(m)(t')$. Let $(V,\rho)$ be an irreducible projective unitary representation. Since the restriction of $\rho$ to $T$ is an ordinary representation, we can decompose $V$ by the weight of $T$ as $V=\oplus V_n$. Let us choose a weight $\lambda$ such that $V_\lambda\neq0$ and take a unit vector $v_\lambda\in V_\lambda$. Let us determine $V_n$ which contains $\rho(m)(v_\lambda)$, by computing $\rho(t)\rho(m)(v_\lambda)$.
\begin{eqnarray*}
\rho(t)\rho(m)(v_\lambda) &=& \rho((t,m))(v_\lambda) \\
&=& \bigl(\kappa^\tau(m)(t)\bigr)^{-1}\rho(m)\rho(t)(v_\lambda) \\
&=& \bigl(\kappa^\tau(m)(t)\bigr)^{-1}\rho(m)\lambda(t)(v_\lambda) \\
&=& [\lambda-\kappa^\tau(m)](t)\rho(m)(v_\lambda).
\end{eqnarray*}
It tells us that $\rho(m):V_\lambda\to V_{\lambda-\kappa^\tau(m)}$. Therefore $V$ has an invariant sub space $\oplus_{m\in\Pi_T}\rho(m)\bb{C}v_\lambda$ and hence it coincides with $V$ since it is irreducible. Since $V$ is completely determined by the orbit $\lambda+\kappa^\tau(\Pi_T)$, we use the symbol $V_{[\lambda]}$ for such a representation. We finish the study of representations of $T\times\Pi_T$ at level $\tau$.

For the third step, recall the definition of the norm of the group $C^*$-algebra.
The norm of $C^*_{{\rm Max}}((T\times\Pi_T)^*)$ is defined by
$$\|f\|:=\sup_{\rho\in\widehat{(T\times\Pi_T)^\tau}}\{\|\pi_\rho(f)\|_{{\rm op}}\},$$
where $\pi_\rho$ is the natural representation of $C_c^\infty((T\times\Pi_T)^\tau)$ induced by $\rho$. $\widehat{(T\times\Pi_T)^\tau}$ denotes the set consisting of isomorphism classes of irreducible unitary representations of $(T\times\Pi_T)^\tau$.
Therefore, it is enough to know the induced representation to study the norm.

Before that, we can simplify the problem.
We claim that $C^\infty_c(T\times\Pi_T,l\tau)$ acts on $(V,\rho)$ at level $k$ trivially unless $k=l$. We verify that $\pi_\rho(f)(v)=0$ for any $f\in C^\infty_c(T\times\Pi_T,l\tau)$. We regard $f$ as a function from $T\times\Pi_T$ and extend to $(T\times\Pi_T)^\tau$ in the obvious manner.
\begin{eqnarray*}
\pi_\rho(f)(v) &=& \int f(t,m)z^{-l}\rho(t,m,z)(v)dtdmdz \\
&=& \int f(t,m)z^{k-l}\rho(t,m,1)(v)dtdmdz \\
&=& \int f(t,m)\rho(t,m,1)(v)dtdm\int z^{k-l}dz
\end{eqnarray*}
It is zero if $k-l\neq 0$. In fact, the twisted convolution algebra comes from such observation. It implies the following result:
$$\|f\|=\|\sum f_kz^k\|=\sup_k \{\|f_k\|\}.$$
Moreover, $\|f_k\|$ can be defined by using only representations at level $k$.
In short,
$$C^*((T\times\Pi_T)^\tau)\cong\bigoplus_{k\in\bb{Z}} C^*(T\times\Pi_T,k\tau)$$
as $C^*$-algebras.
Let us notice that $(T\times\Pi_T)^\tau$ is amenable and hence we do not need subscripts neither ${\rm Max}$ nor ${\rm min}$ for the group $C^*$-algebra.



Let us study the induced action of $C_c^\infty(T\times\Pi_T,k\tau)$ on $V_{[\lambda]}$. $C_c^\infty(T\times\Pi_T,k\tau)$ has a good basis $\{\delta_m\theta^l\}$ defined by
$\delta_n\theta^l(t,m):=\delta_{n,m}t^l$
for $n\in\Pi_T$ and $l\in\bb{Z}$, where $\delta_{n,m}$ is Kronecker's delta. If one deals with a general torus, one should replace $\theta^l$ with a character of $T$. We compute two objects; $\delta_n\theta^l*_{k}\delta_{n'}\theta^{l'}$ and $\pi_\rho(\delta_n\theta^l)(v)$.

Let us compute the twisted convolution in a similar technique in the above.
\begin{eqnarray*}
&&\delta_n\theta^l*_{k}\delta_{n'}\theta^{l'}(t,m) \\
&&\;\;\;\;\;\;\;\;\; =\int \delta_n\theta^l(t',m')\delta_{n'}\theta^{l'}(t'^{-1}t,m'^{-1}m)\kappa^\tau(m')(t't^{-1})^{-k}dt'dm' \\
&&\;\;\;\;\;\;\;\;\; =\int \delta_{n,m'}\delta_{n',m'^{-1}m}t'^{l-l'}t^{l'}\bigl(\kappa^\tau(m')(t^{-1}t')\bigr)^{-k}dt'dm' \\
&&\;\;\;\;\;\;\;\;\; =\delta_{n+n',m}\int\bigl( -k\kappa^\tau(n)+l-l')(t')dt' \\
&&\;\;\;\;\;\;\;\;\; =
\begin{cases}\delta_{n+n'}\theta^{l'-\kappa^\tau(n)}(t,m) & l=l'+k\kappa^\tau(n) \\
0 & \text{otherwise.} \end{cases}
\end{eqnarray*}
It tells us that if $l$ and $l'$ belong different orbits, the product is always $0$. Therefore we can decompose $C^*(T\times\Pi_T,-k\tau)$ as the direct sum
$$\ca{K}_{[\lambda]_k}:=\overline{C_c(\Pi_T)\otimes {\rm Span}_{\bb{C}}\{ \theta^l\mid l\in\lambda+k\kappa^\tau(\Pi_T)\}}$$
as {\bf algebras}, that is, the product of each element of $\ca{K}_{[\lambda]_k}$ and one of $\ca{K}_{[\lambda']_{k}}$ is always $0$ unless $[\lambda]_{k}=[\lambda']_{k}$.

Let us move to the calculation of $\pi_\rho(\delta_n\theta^l)(v)$ for $v\in V_\lambda$.
\begin{eqnarray*}
\pi_\rho(\delta_n\theta^l)(v) &=& \int \delta_n\theta^l(t,m)\rho((t,m))(v)dtdm \\
&=& \int \delta_{n,m}t^l [\lambda-k\kappa^\tau(m)](t)(\rho(m)(v)) dtdm \\ 
&=&\begin{cases}
\rho(m)(v) & \lambda-k\kappa^\tau(n)+l=0 \\
0 & \text{otherwise.} \end{cases}
\end{eqnarray*}
Therefore we can create any finite rank operators preserving the dense subspace obtained by the algebraic direct sum $\oplus^{{\rm alg}}V_{\lambda+k\kappa^\tau(n)}$. 
Since the linear space spanned by $\delta_n\theta^l$'s is dense in $C_c^\infty(T\times\Pi_T,k\tau)$, and the representation $\pi_\rho:C_c^\infty(T\times\Pi_T,k\tau)\to\ca{B}(V_{[\lambda]_k})$ is continuous, the image of the representation is contained in $\ca{K}(V_{[\lambda]_k})$. From the definition of the group $C^*$-algebra, the image of $C^*(T\times\Pi_T,k\tau)$ is the whole of $\ca{K}(V_{[\lambda]_k})$.
Therefore we obtain the isomorphism
$$\ca{K}_{[\lambda]_k}\cong\ca{K}(V_{[\lambda]_k}).$$

Combining all of them, we reach the conclusion.
\end{proof}

\begin{cor}
$$K_0(C^*(T\times\Pi_T,\tau))\cong \bigoplus_{\Lambda_T/\kappa^\tau(\Pi_T)}\bb{Z}.$$
It is isomorphic to $R^\tau(T\times\Pi_T)$, which is the Grothendieck completion of the semi group consisting of $\tau$-twisted representations of $T\times\Pi_T$.
\end{cor}

The action of $\pi_\rho$ computed above tells us that each summand $C^*(T\times\Pi_T,n\tau)$ for $n\neq1$ acts on a $\tau$-twisted representation space $V_\lambda$ trivially for $\tau$-twisted representation $V_{[\lambda]}$. Since our index $\ker(\cancel{D}^+)-\ker(\cancel{D}^-)$ is $\tau$-twisted, it is an element of the direct summand $K_0(C^*(T\times\Pi_T,\tau))$. We reach the following desired result.
\begin{cor}
$$\ind(\cancel{D})\in K_0(C^*(T\times\Pi_T,\tau))\cong R^\tau(T\times\Pi_T).$$
\end{cor}

\subsection{$C^*(\Omega T_0,\tau)$, $C^*(LT,\tau)$ and its noncommutative topology}
We also use symbols in the previous section here. We construct $C^*(\Omega T_0,\tau)$ by the method of inductive limit. To carry it out, we use the following.
\begin{thm}
$$C^*(U_N,\tau)\cong \ca{K}(L^2(\bb{R}^N)).$$
\end{thm}

Therefore we can construct an inductive system $\cdots \to C^*(U_N,\tau)\to C^*(U_{N+1},\tau)\to\cdots$ by tensoring with $P$ which is the projection onto the $1$-dimensional space spanned by the lowest weight vector.
\begin{proof}
For simplicity, we verify the result when $N=1$. We can use an explicit description of representation to study the action of the $C^*$-algebra. Let $g=(a,b)\in U_1$, $\phi\in L^2(\bb{R})$ and $x\in\bb{R}$. 
The representation $\rho$ is given by
$$\rho_g(\phi)(x):=e^{i(2bx+ab)}\phi(x+a)$$
and hence the action of a rapidly decreasing $f$ is given by
$$\pi_\rho(f)(\phi)(x):=\int f(a,b)e^{i(2bx+ab)}\phi(x+a)dadb.$$
Let us verify that a certain dense subspace of $\mathscr{S}(U_1)$ is mapped to $\ca{K}(L^2(\bb{R}))$ and its image is also dense, where $\mathscr{S}(U_1)$ is the space of rapidly decreasing functions. If it is true,  $C^*(\Omega T_0,\tau)$ is $\ca{K}(L^2(\bb{R}))$ itself, since $C^*(U_1,\tau)$ is defined through the representation $\pi_\rho$ induced from the unique irreducible representation $\rho$ of $U_1$.

\begin{lemma}
Each of functions of the form
$$(\text{polynomial of }g)e^{-\frac{1}{2}|g|^2}$$
acts on $L^2(\bb{R})$ as a finite rank operator preserving finite energy vectors.
\end{lemma}
Let us verify it. One can verify that $e^{-\frac{1}{2}|g|^2}$ is a rank one projection onto the space spanned by the lowest weight vector up to constant by computation. Creation/annihilation operators acting on $\mathscr{S}(U_1)$ and ones acting on $L^2(\bb{R})$ are related to one another as follows. We show only one example and the readers can enjoy the game to verify the rest.
\begin{eqnarray*}
\frac{d}{d x}\pi_\rho(f)(\phi)(x) &=& \frac{d}{dx}\int f(a-x,b)e^{i(bx+ab)}\phi(a)dadb \\
&=& \int \bigl( -\frac{\partial f}{\partial a}(a-x,b)+ibf(a-x,b)\bigr)e^{i(bx+ab)}\phi(a)dadb \\
&=& \pi_\rho(-\frac{\partial f}{\partial a}+ibf)(\phi(x).
\end{eqnarray*}

\end{proof}
The above implies that
$$C^*(\Omega T_0,\tau):=\varinjlim C^*(U_N,\tau)\cong\varinjlim \ca{K}(L^2(\bb{R}^N))\cong \ca{K}(L^2(\bb{R}^\infty)).$$

Let us define the twisted group $C^*$-algebra of $LT$ by the limit
$$C^*(LT,\tau):=\varinjlim C^*(T\times\Pi_T\times U_N,\tau)$$
and it is isomorphic to $C^*(T\times\Pi_T,\tau)\otimes C^*(\Omega T_0,\tau)\cong C^*(T\times\Pi_T,\tau)\otimes\ca{K}(L^2(\bb{R}^\infty)).$
We have already known the structure of it. In particular, we can compute its $K$-group easily.
\begin{cor}
$$K_0(C^*(LT,\tau))\cong \bigoplus_{[\lambda]\in\Lambda_T/\kappa^\tau(\Pi_T)}\bb{Z}\cong R^\tau(LT).$$
\end{cor}
It is a ``$C^*$-algebraic proof'' of Proposition \ref{Prop RLT Sect.2}.

\section{The function algebra and the spectral triple}\label{Section Algebra}
Firstly, we construct a $C^*$-algebra which we regard as ``$LT\ltimes_\tau C_0(\ca{M})$.'' As a result, it is the tensor product of $(T\times\Pi_T)\ltimes_\tau C_0(\widetilde{M})$ with $\ca{K}(L^2(\Omega T_0))$. However, our construction is not just a tensor product, but an inductive limit explained in Section \ref{Section simplification}.

Secondly, we combine the algebra with the construction in Section \ref{Section Hilbert} in terms of spectral triples following \cite{CGRS}; we study summability of the spectral triple. We find that our spectral triple is not finitely summable. It tells us that our triple comes from an infinite dimensional object in the sense of noncommutative geometry.

\subsection{Infinite dimensional crossed products}
$L^2(U_{N})$ is one of the most fundamental representation spaces of the $C^*$-algebra $C_0(U_N)$, which acts on by the pointwise multiplication. Moreover, $U_N$ acts on this Hilbert space projectively. More precisely, the action of $g\in U_N$ on $\phi\in L^2(U_N)$ is defined by the formula which we have already used sometimes:
$$[g.\phi](x):=\phi(x-g)\tau(g,x).$$
Clearly, it is a $\tau$-twisted $U_N$-$C_0( U_N)$-bimodule by the following computation.
Let $g,x\in U_N$, $f\in C_0( U_N)$ and $\phi\in L^2( U_N)$.
\begin{eqnarray*}
[g.(f\cdot\phi)](x) &=& (f\cdot \phi)(x-g)\tau(g,x) \\
&=& f(x-g)\cdot\bigl(\phi(x-g)\tau(g,x)\bigr) \\
&=& [g.f](x)[g.\phi](x) \\
&=&(g.f)\cdot (g.\phi)(x)
\end{eqnarray*}
Notice that the action of $U_N$ on $C_0(U_N)$ is not twisted.
\begin{thm}
The above action defines an isomorphism
$$ U_N\ltimes_\tau C_0( U_N)\cong \ca{K}(L^2( U_N)).$$
Using the lowest weight vector, we can define a natural rank $1$-projection $P$ and we can define a homomorpshism 
$$ U_N\ltimes_\tau C_0( U_N)\to U_{N+1}\ltimes_\tau C_0( U_{N+1})$$
by the formula 
$$k\mapsto k\otimes P$$
under the isomorphism $L^2( U_{N+1})\cong L^2( U_N)\otimes L^2(U_1)$.
\end{thm}
\begin{proof}
We verify here that the above representation is faithful, valued in $\ca{K}(L^2(U_N))$, and the image is dense. If they are true, from the injectivity, the representation is isometric. Combining it with the denseness of the image, the representation is surjective. 

Let $\phi\in L^2(U_N)$, $x\in U_N$ a point of the space, $g\in U_N$ an element of the group $U_N$, and $f\in \mathscr{S}(U_N\times U_N)$. We regard $f$ as an element of $U_N\ltimes_\tau C_0(U_N)$.
By a change of variables, we obtain the following useful form.
\begin{eqnarray*}
(f.\phi)(x) &=& \int f(g,x)\phi(x-g)\tau(g,x) dg \\
&=& \int f(x-g,x)\tau(x,g)\phi(g)dg.
\end{eqnarray*}
If $f$ satisfies the equation $f(g-x,x)\tau(g,x)=f_1(x)\overline{f_2(g)}$
for some rapidly decreasing $f_1$ and $f_2$,
the operator induced by $f$ is a Schatten form. Since the space consisting of linear combinations of such functions is dense in $\mathscr{S}(U_N\times U_N)$, the image of the representation is contained in $\ca{K}(L^2(U_N))$.
Moreover, since $\mathscr{S}(U_N)$ is dense in $L^2(U_N)$, and the space consisting of linear combinations of Shatten forms is dense in $\ca{K}(L^2(U_N))$, the homomorphism has the dense range.

The injectivity can be verified as follows.
Let $F(g,x):=f(x-g,x)\tau(x,g)$ for simplicity. Then
$$(f.\phi)(x)=\int F(g,x)\phi(g)dg$$
In order to find $\phi$ such that $f.\phi\neq0$ for $f\neq0$, we fix a point $p\in U_N^2$ such that $F(p)\neq0$ and write $p=(p_1,p_2)\in U_N^2$. Let $b$ be a bump function on $U_N$ supported around $p_2$. $\phi(x):=b(x)\overline{F(p_1,x)}$ is a function we need.

Let $P$ be the Schatten form $\Omega\otimes\Omega^*$, where $\Omega$ is the lowest weight vector. Using it, we can define a homomorphism
$$U_N\ltimes C_0( U_N)\to U_{N+1}\ltimes C_0( U_{N+1})$$
by $k\mapsto k\otimes P$.


\end{proof}

In Section \ref{Section Hilbert}, we justified the WZW model $L^2(\Omega T_0)\cong \ca{H}_{{\rm WZW}}$ and it can be written as the inductive limit
$L^2(\Omega T_0):=\varinjlim L^2(U_N).$
Under this notation, $\Omega T_0\ltimes_\tau C_0(\Omega T_0)\cong \ca{K}(L^2(\Omega T_0))$ and it is Morita equivalent to $\bb{C}$.


Combining the above with Proposition \ref{Prop division Sect.3},
we know the structure of our $C^*$-algebra.
\begin{cor}
$$LT\ltimes_\tau C_0(\ca{M})\cong (T\times\Pi_T)\ltimes_\tau C_0(\widetilde{M})\bigotimes \ca{K}(L^2(\Omega T_0)).$$
It is Morita equivalent to $(T\times\Pi_T)\ltimes_\tau C_0(\widetilde{M})$, which is a $C^*$-algebraic description of the local equivalence of groupoids $\ca{M}//LT\cong \widetilde{M}//(T\times\Pi_T)$.
\end{cor}

\subsection{The spectral triple and its summability}
Let us set
$$A:=LT\ltimes_\tau C_0(\ca{M})$$
$$\ca{H}:=L^2(\ca{M},\ca{L}\otimes\ca{S}).$$

Since $(T\times \Pi_T)\ltimes_\tau C_0(\widetilde{M})$ acts on $L^2(\widetilde{M},\ca{L}\otimes S(T\widetilde{M}))$ by the natural action, and $\Omega T_0\ltimes_\tau C_0(\Omega T_0)$ acts on $\ca{H}_{{\rm WZW}}\otimes S$ through the isomorphism $\Omega T_0\ltimes_\tau C_0(\Omega T_0)\cong \ca{K}(L^2(\Omega T_0))=\ca{K}(\ca{H}_{{\rm WZW}})$, we consider that $A$ is a subalgebra of $\ca{B}(\ca{H})$.

In order to get a smooth algebra in $A$, we set some dense subspaces:
\begin{itemize}
\item $L^2(\bb{R}^\infty)_\fin$ is the set of finite energy vectors.
\item $\ca{H}_{\WZW,\fin}:=L^2(\bb{R}^\infty)_\fin^*\otimes^\alg L^2(\bb{R}^\infty)_\fin$ the algebraic tensor product.
\item $\ca{H}_\fin:=C_c^\infty(\widetilde{M},\ca{L}\otimes S(T\widetilde{M}))\bigotimes^\alg \ca{H}_{\WZW,\fin}\otimes^\alg S_\fin$.
\item $A_\fin$ is the subalgebra of $A$ consisting of $a\in A$ preserving $\ca{H}_\fin$.
\end{itemize}
We take a completion of $A_{{\rm fin}}$ with respect to the following sequence of norms to obtain a smooth algebra. Notice that $A_\fin$ and $\ca{D}$ preserves $\ca{H}_{\fin}$.
\begin{dfn}We call a norm
$$\|a\|_k:=\sum_{j=0}^k\|\ad(\ca{D})^j(a)\|_{{\rm op}}$$
a $C^k$ norm and we set Banach algebras associated with $A$
$$\overline{\ca{A}}^{\|\cdot\|_k}:=\overline{A_{{\rm fin}}}^{\|\cdot\|_k}$$
and a Fr\'{e}chet algebra
$$\ca{A}:=\bigcap_{k=0}^\infty\overline{\ca{A}}^{\|\cdot\|_k}.$$
\end{dfn}
We verify that $\ca{A}$ is a smooth algebra and $(\ca{A},\ca{H},\ca{D})$ is a spectral triple. For this purpose, we prepare the followings.

\begin{pro}
$\ca{A}$ is a smooth algebra, that is, $\ca{A}$ is Fr\'{e}chet, $*$-closed and the unitalization of $\ca{A}$ is stable under the holomorphic functional calculus.
\end{pro}
\begin{proof}
It is sufficient to verify the stability under the holomorphic functional calculus, that is, we need to verify that
$$\omega(a)=\frac{1}{2\pi i}\oint\frac{\omega(z)}{z-a}dz$$
is an element of $\ca{A}$ for any
$a\in \ca{A}$ and a holomorphic function $\omega$ from a certain neighborhood of $\sigma(a)$ to $\bb{C}$.

We need the following purely functional analytical lemma.
Let $B(\varepsilon):=\{z\in\bb{C}\mid|z|<\varepsilon\}$ and $K+B(\varepsilon):=\{k+z\mid k\in K,z\in B(\varepsilon)\}$.
\begin{lemma}
For any $\varepsilon$, there exists $\delta>0$ such that
\begin{center}
``$\|a-a'\|_{{\rm op}}<\delta$'' implies ``$\sigma(a')\subseteq \sigma(a)+B(\varepsilon)$.''
\end{center}
\end{lemma}
\begin{proof}
We set a constant $C$ as follows:
$$C:=\max\{\|(a-\lambda)^{-1}\|_{{\rm op}}\mid d(\lambda,\sigma(a))\geq\varepsilon\}.$$
$C$ is indeed a finite number because $\|(a-\lambda)^{-1}\|_{{\rm op}}$ vanishes at infinity and the set of $\lambda$ such that $d(\lambda,\sigma(a))\geq\varepsilon$ is closed. $\delta:=C^{-1}/2$ is the number which we are seeking. In fact, suppose that $\|a-a'\|<\delta$ and $\lambda\notin \sigma(a)+B(\varepsilon)$, then $a'-\lambda$ is invertible thanks to Neumann series argument, because
$$a'-\lambda=a-\lambda+a'-a=(a-\lambda)(1+(a-\lambda)^{-1}(a'-a)),$$
and $(a-\lambda)^{-1}(a'-a)$ is small enough.
Therefore $\lambda\notin \sigma(a')$.
\end{proof}
Let $\Gamma$ be an integral path for the above contour integral. Since the both of $\Gamma$ and $\sigma(a)$ are compact, $d(\Gamma,\sigma(a))=:2\varepsilon>0$ and hence there exists $\delta$ as above. We can take an approximating sequence $\{a_n\}\subseteq A_{{\rm fin}}$ converging to $a$ and satisfying that $\|a-a_n\|<\delta$. It tells us that $\omega(a_n)$ can be defined by the same path.

$\omega(a_n)$ converges to $\omega(a)$ as follows.
\begin{eqnarray*}
\|\omega(a_n)-\omega(a)\| &=&\|\frac{1}{2\pi i}\oint\omega(z)\Bigl(\frac{1}{z-a_n}-\frac{1}{z-a}\Bigr)dz\| \\
&=& \frac{1}{2\pi}\|\oint\omega(z)\Bigl(\frac{1}{z-a_n}(z-a-z+a_n)\frac{1}{z-a}\Bigr)dz\|\\
&\leq& \frac{1}{2\pi}|\Gamma|\sup{|\omega(z)|}\sup_\Gamma{\|\frac{1}{z-a_n}\|}\|a_n-a\|\sup_\Gamma{\|\frac{1}{z-a}\|},
\end{eqnarray*}
where $|\Gamma|$ is the length of the path. Since $a_n$ converges to $a$, $\frac{1}{z-a_n}$ also converges to $\frac{1}{z-a}$, hence $\frac{1}{z-a_n}$ is a bounded sequence and $\omega(a_n)$ converges to $\omega(a)$.

${\rm ad}(\ca{D})(\omega(a_n))$ converges to $\rm{ad}(\ca{D})(\omega(a))$ as follows.


\begin{eqnarray*}
&&\|{\rm ad}(\ca{D})(\omega(a_n))-{\rm ad}(\ca{D})(\omega(a))\| \\
&&\;\;\;=\|\frac{1}{2\pi i}\oint\omega(z){\rm ad}(\ca{D})\Bigl(\frac{1}{z-a_n}(a_n-a)\frac{1}{z-a}\Bigr)\| \\
&&\;\;\;= \frac{1}{2\pi}\|\oint\omega(z)\Bigl(\frac{1}{z-a_n}[\ca{D},z-a_n]\frac{1}{z-a_n}
(a_n-a)\frac{1}{z-a}+ 
\\
&& \;\;\;\;\;\; \frac{1}{z-a_n}[\ca{D},a_n-a]\frac{1}{z-a}+\frac{1}{z-a_n}(a_n-a)\frac{1}{z-a}[\ca{D},z-a]
\frac{1}{z-a}
\Bigr)dz\|\\
&&\;\;\; \leq |\Gamma|\frac{1}{2\pi}\sup{|\omega(z)|}\sup{\|\frac{1}{z-a_n}\|^2}\sup{\|\frac{1}{z-a}\|}\|[\ca{D},a_n]\|\|a_n-a\|+\text{similar ones}
\end{eqnarray*}
The higher derivatives ${\rm ad}(\ca{D})^k(\omega(a_n))$ converges in the same manner. Therefore, $\{\omega(a_n)\}$ is an approximating sequence of $\omega(a)$ in the sense of $\ca{A}$.
\end{proof}

Let us move to the verification of the three conditions of spectral triples. Before that, we precisely describe the domain of $\ca{D}$.

\begin{dfn}We call the inner product on $\ca{H}_{{\rm fin}}$
\begin{eqnarray*}
\innpro{\phi}{\psi}{H^1}&:=&\innpro{\phi}{\psi}{}+\innpro{\ca{D}\phi}{\ca{D}\psi}{} \\
&=& \innpro{\phi}{(1+\ca{D}^2)\psi}{}\\
&=& \innpro{\phi}{\psi}{}+\innpro{|\ca{D}|\phi}{|\ca{D}|\psi}{}
\end{eqnarray*}
the $H^1$-inner product. The domain of $\ca{D}$ can be obtained by using it as
$${\rm dom}(\ca{D})=\overline{\ca{H}_{{\rm fin}}}^{\|\cdot\|_{H^1}}.$$
\end{dfn}

\begin{lemma}
$\overline{\ca{A}}^{\|\cdot\|_1}$ preserves $\rm{dom}(\ca{D})$,
in particular, $\ca{A}$ preserves ${\rm dom}(\ca{D})$.
\end{lemma}
\begin{proof}
$\phi\in {\rm dom}(\ca{D})$ if and only if there exists an approximating sequence $\{\phi_n\}$ such that it converges to $\phi$ and $\{\ca{D}(\phi_n)\}$ converges to some element in $\ca{H}$.
In the similar way, for $a\in\overline{\ca{A}}^{\|\cdot\|_1}$, we take an approximating sequence $\{a_n\}\subseteq A_{{\rm fin}}$, which converges to $a$ in the $C^1$-norm.

Let us verify that $a\phi\in{\rm dom}(\ca{D})$. It is sufficient to verify that $a_n\phi_n$ is an approximating sequence of $a\phi$ in the sence of $H^1$, that is, $\ca{D}(a_n\phi_n)$ converges to some vector. We prove that it is a Cauchy sequence.
\begin{eqnarray*}
\|\ca{D}(a_n\phi_n)-\ca{D}(a_m\phi_m)\| &=& \| [\ca{D},a_n]\phi_n+a_n\ca{D}(\phi_n)-[\ca{D},a_m]\phi_m-a_m\ca{D}\phi_m\|\\
&=& \|([\ca{D},a_n]-[\ca{D},a_m])\phi_n+[\ca{D},a_m](\phi_n-\phi_m)\\
&& \;\;\;\;+(a_n-a_m)\ca{D}(\phi_n)+a_m(\ca{D}(\phi_n)-\ca{D}(\phi_m)\| \\
&\leq& \|[\ca{D},a_n]-[\ca{D},a_m]\|_{{\rm op}}\|\phi_n\|+\|[\ca{D},a_m]\|_{{\rm op}}\|\phi_n-\phi_m\|  \\
&&\;\;\; +\|a_n-a_m\|_{{\rm op}}\|\ca{D}(\phi_n)\|+\|a_m\|_{{\rm op}}\|\ca{D}(\phi_n)-\ca{D}(\phi_m)\|.
\end{eqnarray*}
By the assumption, all of them are small and hence it is a Cauchy sequence.
\end{proof}

\begin{lemma}
$[\ca{D},a]$ is bounded on ${\rm dom}(\ca{D})$ for $a\in \overline{\ca{A}}^{\|\cdot\|_1}$. In particular it is bounded for $a\in \ca{A}$.
\end{lemma}
\begin{proof}
Let $a\in\ca{A}$ and $\{a_n\}$ be an approximating sequence of it. Then the sequence $[\ca{D},a_n]$ converges to $[\ca{D},a]$ by the definition of the topology of $\overline{A}^{\|\cdot\|_k}$. Since the algebra $\ca{B}(\ca{H})$ is complete, we obtain the result.
\end{proof}

\begin{lemma}
$a(1+\ca{D}^2)^{-1}$ is compact for $a\in A$.
\end{lemma}
\begin{proof}
It suffices to verify the statement when $a\in A_{{\rm fin}}$, that is, $a$ is of the form $\sum f_i\otimes k_i$, where $f_i\in C_c^\infty(T\times\Pi_T,C_c^\infty(\widetilde{M}))$ and $k_i\in \ca{K}(\ca{H}_{{\rm WZW}})\otimes{\rm id}_S$.
Therefore, we can separate the argument into two parts, the $\widetilde{M}$-part and the $\Omega T_0$-part.

$f_i(1+\cancel{D}^2)^{-1}$ is compact as usual.

$k_i(1+\cancel{\partial}^2)^{-1}$ is compact as follows.
Let us notice that $k_i$ can be approximated by a sum of operators of the form 
\begin{center}
compact operator $\otimes$ compact operator $\otimes$ id
\end{center}
and $(1+\cancel{D}^2)^{-1}$ is a sum of operators of the form
\begin{center}
id $\otimes$ compact operator $\otimes$ compact operator
\end{center}
and hence the composition of them is compact.

\end{proof}

Combining all of them, we obtain the following.
\begin{thm}
The triple $(\ca{A},\ca{H},\ca{D})$ is a spectral triple.

\end{thm}
Let us verify the following.

\begin{pro}
Our spectral triple $(\ca{A},\ca{H},\ca{D})$ has infinite spectral dimension.
\end{pro}

\begin{proof}
It is enough to deal with $\cancel{\partial}$ for this purpose. We verify that
$$P\otimes {\rm id}\circ(1+\cancel{\partial}^2)^{-s}$$
is not a trace class operator for any $s>0$, where $P$ is the rank 1 projection onto the lowest weight vector constructed above.

We can take $\{v_i^*\otimes v_j\otimes s_k\}$ as a C.O.N.S. of $\ca{H}=\ca{H}_{{\rm WZW}}\otimes S$, where $\{v_i^*\otimes v_j\}$ is a C.O.N.S. of $\ca{H}_{{\rm WZW}}$ consisting of eigenvectors of $d\rho^*(d)$, and $\{s_k\}$ is a C.O.N.S. of $S$ consisting of eigenvectors of the number operator $N$. Eigenvalues of $v_i^*$ ($s_k$) is written as $-|v_i^*|$ ($|s_k|$ respectively). We use here a notation such that $|\cdot|$ is non negative. Recall that $\rho^*$ is a negative energy representation.
Let us notice that
$$P\otimes {\rm id}\circ(1+\cancel{\partial}^2)^{-s}v_i^*\otimes v_j\otimes s_k=P\otimes {\rm id}\circ[(1+|v_i^*|+|s_k|)^{-s}v_i^*\otimes v_j\otimes s_j]$$
$$=\begin{cases} [(1+|s_k|)^{-s}v_i^*\otimes v_j\otimes s_k] & v_i^*\otimes v_j\text{ is the lowest weight vector} \\
0 & \text{otherwise}. \end{cases}$$

Let us recall that $N(\overline{z_{j_1}}\wedge\overline{z_{j_2}}\wedge\cdots\wedge\overline{z_{j_n}})=(\sum_{k=1}^n j_k)\overline{z_{j_1}}\wedge\overline{z_{j_2}}\wedge\cdots\wedge\overline{z_{j_n}}$ and the dimension of $n$-eigenspace $\ker(N-n)$ is the distinct partition number $q(n)$ and it is estimated as
$$q(n)\geq C\exp(\frac{1}{2}n^{\frac{1}{4}})$$
due to the following lemma.

It enable us to estimate the trace of $P(1+\cancel{\partial}^2)^{-s}$ as follows.
$${\rm Trace}(P(1+\cancel{\partial}^2)^{-s}) \geq \sum C(1+n)^{-s}\exp(\frac{1}{2}n^{\frac{1}{4}})=\infty.$$

\end{proof}

We verify an estimate about the distinct partition function we used above. For this purpose, we need the estimate due to \cite{HR};
$p(n)\geq C\exp(\sqrt{n}).$
\begin{lemma}
$$q(n)\geq C\exp(\frac{1}{2}n^{\frac{1}{4}}).$$
\end{lemma}
\begin{proof}
We introduce some notations.
Let $\ca{P}_n$ be the set of partition of $n$, and $\ca{Q}_n$ be the set of distinct partition of $n$. Then, $p(n)=\#\ca{P}_n$ and $q(n)=\#\ca{Q}_n$. We define an injection $\ca{P}_l\to \ca{Q}_{l+\frac{1}{2}l(l-1)}$ by the formula
\begin{eqnarray*}
(1,1,\cdots,1) &\mapsto& (l,l-1,\cdots,1)\\
(\lambda_1,\lambda_2,\cdots\lambda_k) &\mapsto& (\lambda_1+l-1,\lambda_2+l-2,\cdots,\lambda_k+l-k,0+n-k-1,\cdots,0+1,0).
\end{eqnarray*}
Therefore $p(l)\leq q(\frac{1}{2}l(l+1))\leq q(l^2)$. In particular, 
$$q(n)\geq p(\lfloor \sqrt{n} \rfloor)\geq C\exp(\sqrt{\lfloor \sqrt{n} \rfloor})\geq C\exp(\frac{1}{2}n^{\frac{1}{4}}).$$

\end{proof}

\section{Applications}\label{Section Applications}


In this section, we present some additional circumstantial evidences that our construction knows some information of $LT\curvearrowright \ca{M}$. There is no hope to establish Gelfand's theorem for infinite dimensional space, since the topological space obtained by Gelfand transformation is always locally compact. Therefore what we can do is only presenting circumstantial evidences or showing nice examples as many as possible.

We have presented circumstantial evidences so far. For example, the isomorphism $LT\ltimes_\tau C_0(\ca{M})\cong (T\times\Pi_T)\ltimes_\tau C_0(\widetilde{M})\bigotimes \ca{K}(L^2(\Omega T_0))$ can be a $C^*$-algebraic proof of the local equivalence $\ca{M}//LT^\tau\cong \widetilde{M}//(T\times\Pi_T)^\tau$; a group $C^*$-algebra $C^*(LT,\tau)$ obtained by a similar method; $L^2(\Omega T_0)\cong L^2(\bb{R}^\infty)^*\otimes L^2(\bb{R}^\infty)$ can be obtained by two different methods, infinite dimensional Peter-Weyl and inductive limit of finite dimensional Euclidean spaces; $K$-group of $C^*(LT,\tau)$ is also an evidence. We add some such applications.

\begin{thm}
For any character $\lambda$ of $T^\tau$, we can define a $LT^\tau$-equivariant line bundle $\ca{L}_\lambda$ on the flag manifold
$\Omega T=LT^\tau/T^\tau$ by
$$\ca{L}_\lambda:=LT^\tau\times_{T^\tau}\bb{C}_{\lambda},$$
and a Dirac operator $\ca{D}_\lambda$ acting on $L^2(\Omega T,\ca{L}_\lambda\otimes\ca{S})$. Moreover, the index of $\ca{D}_\lambda$ is the irreducible P.E.R. $V_{[\lambda]}$.
\end{thm}
This theorem is completely parallel to well-known Borel-Weil theorem.
\begin{proof}
It is almost clear from the construction of the $L^2$-space. Our Hilbert space is
$$L^2(\Omega T,\ca{L}\otimes\ca{S})=L^2(\Pi_T)\bigotimes\ca{H}_{\WZW,n}\otimes S,$$
since the locally compact model $\Pi_T$ is $0$-dimensional. The index for $\cancel{D}=0$ is the representation space $L^2(\Pi_T)$ itself. 

Let us consider the restriction of $\ca{L}_\lambda$ to $\Pi_T$. Although it is the quotient space $(T\times\Pi_T)^\tau\times_{T^\tau}\bb{C}_{-\lambda}$, we can trivialize as $\Pi_T\times\bb{C}$. We write an element $[(t_0,n_0),v]$ of $(T\times\Pi_T)^\tau\times_{T^\tau}\bb{C}_{-\lambda}$ by using the trivialization $[(1,n_0),\lambda(t_0)^{-1}v]$. Let us calculate the actions.
\begin{eqnarray*}
t.[(1,n_0),v] &=& [(t,n_0),v]\\
&=& [(1,n_0),\lambda(t)v] \\
&=& \lambda(t)[(1,n_0),v], \\
n.[(1,n_0),v] &=& [(1,n+n_0),v]
\end{eqnarray*}
Considering the definition of the action on the functions, we obtain that $L^2(\Pi_T,\ca{L}|_{\Pi_T})$ is a generator corresponding to $V_{[\lambda]}$. 
\end{proof}
The following is just a corollary of our main theorem.
\begin{pro}
Let $\ca{M}$ be an $LT$-Hamiltonian space. 
The symplectic form $\Omega$ determines a $\tau$-twisted $LT$-equivariant pre-quantum line bundle $\ca{L}$. The index for $\ca{L}$ is an element of $R^\tau(LT)$.
\end{pro}
It is competely parallel to Bott's quantization.

\section*{Acknowledgements}
I am very grateful to my supervisor Nigel Higson. He gives me many useful suggestions and in particular he stimulated my interest to \cite{Son}. I also thank my supervisor Tsuyoshi Kato, his students and former ones. I benefit from their various back grounds. In particular, I received the paper \cite{Kas} from Yoshiyasu Fukumoto. I am supported by JSPS KAKENHI Grant Number 16J02214 and
the Kyoto Top Global University Project (KTGU).

\end{document}